\documentclass[a4paper,12pt]{amsart}
\usepackage{amsmath,amsthm,amssymb}
\usepackage[all]{xy}
\usepackage[top=30truemm,bottom=30truemm,left=25truemm,right=25truemm]{geometry}
\usepackage{ytableau}

\newtheorem{theo}{Theorem}[subsection]
\newtheorem*{thm}{Theorem}
\newtheorem{defi}[theo]{Definition}
\newtheorem{lem}[theo]{Lemma}
\newtheorem{rem}[theo]{Remark}
\newtheorem{prop}[theo]{Proposition}
\newtheorem{cor}[theo]{Corollary}

\newtheorem{ex}[theo]{Example}

\newcommand{\nc}{\newcommand}
\nc{\on}{\operatorname}
\nc{\C}{\mathbb{C}}
\nc{\R}{\mathbb{R}}
\nc{\Q}{\mathbb{Q}}
\nc{\Z}{\mathbb{Z}}
\nc{\N}{\mathbb{N}}

\nc{\bfa}{\mathbf{a}}
\nc{\bfi}{\mathbf{i}}

\nc{\clA}{\mathcal{A}}
\nc{\clB}{\mathcal{B}}
\nc{\clL}{\mathcal{L}}

\nc{\g}{\mathfrak{g}}
\nc{\frS}{\mathfrak{S}}
\nc{\frsl}{\mathfrak{sl}}

\nc{\aff}{\mathrm{aff}}

\nc{\inv}{^{-1}}
\nc{\qu}{\quad}
\nc{\qqu}{\qquad}
\nc{\la}{\langle}
\nc{\ra}{\rangle}

\nc{\Ker}{\on{Ker}}
\nc{\im}{\on{Im}}
\nc{\Hom}{\on{Hom}}
\nc{\End}{\on{End}}
\nc{\Span}{\on{Span}}
\nc{\id}{\on{id}}
\nc{\GT}{\on{GT}}
\nc{\AP}{\on{AP}}
\nc{\RE}{\on{RE}}
\nc{\wt}{\on{wt}}
\nc{\SST}{\on{SST}}
\nc{\Br}{\on{Br}}
\nc{\YB}{\on{YB}}
\nc{\str}{\on{string}}
\nc{\lex}{\on{lex}}

\nc{\ol}{\overline}
\nc{\ul}{\underline}
\nc{\hf}{\frac{1}{2}}
\nc{\vphi}{\varphi}
\nc{\vpi}{\varpi}
\nc{\vep}{\varepsilon}
\nc{\lm}{\lambda}
\nc{\Lm}{\Lambda}
\nc{\eps}{\epsilon}

\nc{\IF}{\text{ if }}
\nc{\AND}{\text{ and }}
\nc{\OW}{\text{ otherwise}}
\nc{\lowerterms}{\text{(lower terms)}}
\nc{\higherterms}{\text{(higher terms)}}
\nc{\Forall}{\text{ for all }}
\nc{\Forsome}{\text{ for some }}
\nc{\For}{\text{ for }}

\nc{\plim}[1][]{\mathop{\varprojlim}\limits_{#1}}
\nc{\ilim}[1][]{\mathop{\varinjlim}\limits_{#1}}

\nc{\Etil}{\widetilde{E}}
\nc{\Ftil}{\widetilde{F}}
\nc{\Itil}{\widetilde{I}}
\nc{\Ltil}{\widetilde{L}}

\title{Alcove paths and Gelfand-Tsetlin patterns}
\author[H. Watanabe and K. Yamamura]{Hideya Watanabe and Keita Yamamura}
\address{(H. Watanabe) Research Institute for Mathematical Sciences, Kyoto University, Kyoto 606-8052, Japan}
\email{hideya@kurims.kyoto-u.ac.jp}
\address{(K. Yamamura) Gifu Prefectural Kaizu Meisei High School, 11-1 Takasu, Kaizu, Gifu 503-0653, Japan}
\email{k-yamamura@ist.osaka-u.ac.jp}
\date{\today}

\begin{document}
\maketitle

\begin{abstract}
In their study of the equivariant K-theory of the generalized flag varieties $G/P$, where $G$ is a complex semisimple Lie group, and $P$ is a parabolic subgroup of $G$, Lenart and Postnikov introduced a combinatorial tool, called the alcove path model. It provides a model for the highest weight crystals with dominant integral highest weights, generalizing the model by semistandard Young tableaux. In this paper, we prove a simple and explicit formula describing the crystal isomorphism between the alcove path model and the Gelfand-Tsetlin pattern model for type $A$.
\end{abstract}


\section{Introduction}
Lenart and Postnikov \cite{LP07} provided a Chevalley-type formula for the equivariant $K$-theory of generalized flag varieties $G/P$, where $G$ is a complex semisimple Lie group, and $P$ is a parabolic subgroup of $G$. Their formula is based on a combinatorial model for the highest weight crystals, called the alcove path model, which was also introduced by them. To be more specific, let $\lm$ be a dominant integral weight, and $A_\circ$ the fundamental alcove. Then, Lenart-Postnikov's Chevalley-type formula tells us that the product of a Schubert class and the class of the line bundle $\clL_\lm$ corresponding to $-\lm$ in the equivariant $K$-theory of $G/P$ is determined by counting the number of certain subsequences (called the admissible subsets) of a fixed sequence of adjacent alcoves (called a reduced alcove path) from $A_\circ$ to $A_\circ - \lm$.

Their formula has applications in the representation theory of $G$, and its Lie algebra $\g$. For each dominant integral weight $\lm$, there exists a unique irreducible highest weight module $V(\lm)$ of highest weight $\lm$. For each integral weight $\mu$, let $V(\lm)_\mu$ denote the weight space of weight $\mu$. The characters $s_\lm := \sum_{\mu} (\dim V(\lm)_\mu) e^\mu$, where the sum runs over the integral weights $\mu$, and $e^\mu$ is the standard basis of the group algebra of the integral weight lattice, play important roles in the representation theory of $\g$. To each admissible subset $J$, an integral weight $\wt(J)$, called its weight, is assigned. Then, Lenart and Postnikov proved the following character formula:
$$
s_\lm = \sum_J e^{\wt(J)},
$$
where the sum runs over the admissible subsets. Since this formula is cancellation-free, the set of admissible subsets (resp., admissible subsets of weight $\mu$) is in one-to-one correspondence with any basis of $V(\lm)$ (resp., $V(\lm)_\mu$).

The irreducible highest weight module $V(\lm)$ has a distinguished basis $B(\lm)$; Lusztig's canonical basis \cite{Lu90}, or Kashiwara's global crystal basis \cite{K91}. In fact, the canonical basis, or global crystal basis is a basis of the irreducible highest weight module $V_q(\lm)$ of highest weight $\lm$ over the quantum group $U_q(\g)$. Taking the limit $q \rightarrow 0$, we obtain the crystal basis $\clB(\lm)$. Although $\clB(\lm)$ is no longer a basis of $V_q(\lm)$ nor $V(\lm)$, it parametrizes the basis elements of $V(\lm)$. Moreover, to each element $b \in \clB(\lm)$, an integral weight $\wt(b)$ is assigned, and we have $s_\lm = \sum_{b \in \clB(\lm)} e^{\wt(b)}$. Hence, there should exist a natural bijection between the set of admissible subsets and $\clB(\lm)$ which preserves the weights.

As we have mentioned above, the crystal basis $\clB(\lm)$ is the limit at $q \rightarrow 0$ of $B(\lm)$, which is a genuine basis of $V_q(\lm)$. Hence, for each $x \in B(\lm)$, the products $E_i x, F_i x \in V_q(\lm)$ make sense, where $E_i,F_i$, $i \in I$ are Chevalley generators of $U_q(\g)$. Kashiwara \cite{K90} defined operators $\Etil_i,\Ftil_i : \clB(\lm) \rightarrow \clB(\lm) \sqcup \{0\}$ which are, roughly speaking, the limits of the actions of $E_i,F_i$ at $q \rightarrow 0$. These maps equip the crystal basis $\clB(\lm)$ with a combinatorial structure, called a crystal structure.

In \cite{LP08}, Lenart and Postnikov defined operators $\Etil_i,\Ftil_i$ on the set of admissible subsets, and proved that these operators together with the map $\wt$ give rise to a structure of crystal isomorphic to $\clB(\lm)$. One feature of the crystal structure of $\clB(\lm)$ is the existence of the highest weight vector. Namely, there exists a special element $b_\lm \in \clB(\lm)$ such that
$$
\clB(\lm) = \{ \Ftil_{i_1} \cdots \Ftil_{i_l}(b_\lm) \mid l \in \Z_{\geq 0},\ i_1,\ldots,i_l \in I \} \setminus \{0\}.
$$
Hence, there exists a special admissible subset $J_\lm$ such that the map $\Ftil_{i_1} \cdots \Ftil_{i_l}(J_\lm) \mapsto \Ftil_{i_1} \cdots \Ftil_{i_l}(b_\lm)$ gives a weight preserving bijection between the set of admissible subsets and the crystal basis $\clB(\lm)$. In fact, it is an isomorphism of crystals.

Now, let us consider the case when $G = \mathrm{SL}_n$. It is well-known that the crystal basis $\clB(\lm)$ is modeled by the set $\SST(\lm)$ of semistandard Young tableaux of shape $\lm$ filled with letters in $\{ 1,\ldots,n \}$. Namely, $\SST(\lm)$ is equipped with a crystal structure in a way such that it is isomorphic to $\clB(\lm)$. Combining the discussion above, we obtain an isomorphism of crystals between the set of admissible subsets and $\SST(\lm)$. Then, it is natural to ask for an explicit description of this isomorphism.

The goal of this paper is to provide a simple and explicit formula describing such an isomorphism. In fact, our formula gives an isomorphism between the set of admissible subsets and the set of Gelfand-Tsetlin patterns of shape $\lm$, the latter of which is in a natural one-to-one correspondence with $\SST(\lm)$.

The crucial point in our strategy is to extend the fixed reduced alcove path. Recall that in the alcove path model, one has to fix a reduced alcove path from $A_\circ$ to $A_\circ - \lm$, and consider its admissible subsets. In this paper, we fix a reduced alcove path $\Pi = (A_\circ=A_0,A_1,\ldots,A_u = w_\circ A_\circ - \lm)$ from $A_\circ$ to $w_\circ A_\circ - \lm$, where $w_\circ$ denotes the longest element of the Weyl group of $G$. After modifying the definition of admissible subsets, we prove that the set of admissible subsets of $\Pi$ is equipped with a crystal structure isomorphic to $\clB(\lm)$.

When $G = \mathrm{SL}_n$, to each admissible subset $J$ of $\Pi$, we can associate a tuple $N(J) = (N_{i,j}(J))_{1 \leq i < j \leq n}$ of nonnegative integers. Here, we omit the precise definition of $N(J)$; see Subsection \ref{Almost decreasing subsets} instead. Then, our main result in this paper is the following:

\begin{thm}
Let $G = \mathrm{SL}_n$, $\lm$ be a dominant integral weight, $\Pi$ a reduced alcove path from $A_\circ$ to $w_\circ A_\circ - \lm$. Let $(\lm_1,\ldots,\lm_n)$ be the partition corresponding to $\lm$. Then, the assignment $J \mapsto (\lm_i-N_{i,j}(J))_{1 \leq i \leq j \leq n}$, where $N_{i,i}(J) := 0$, gives rise to an isomorphism of crystals from the set of admissible subsets and the set of Gelfand-Tsetlin patterns of shape $\lm$.
\end{thm}

In fact, the crystal isomorphism between the alcove path model and the Gelfand-Tsetlin pattern model, or the semistandard tableau model was constructed in \cite{LL15} in more generality. However, we do not use any results in \cite{LL15} in this paper. Also, our formula is easier and more explicit; the isomorphism can be calculated just by counting the root multiplicities in a given admissible subset.

This paper is organized as follows. In Section 2, we prepare basic notions concerning crystals. Especially, we recall the string data of Gelfand-Tsetlin patterns. Section 3 is devoted to reviewing Lenart-Postnikov's alcove path model. In Section 4, we introduce the extended alcove path model, and prove the main theorem by comparing the string data of the admissible subsets with those of the Gelfand-Tsetlin patterns.

\subsection*{Acknowledgement}
HW would like to thank Susumu Ariki for bringing his attention to this topic. The authors are grateful to the referees for careful readings and helpful comments.

\section{Crystals}
In this section, we recall basic notions of (abstract) crystals, especially the highest weight crystals, string data, and the Gelfand-Tsetlin patterns. For details, see e.g., \cite{BS17}. We also refer the readers to \cite{BB05} or \cite{H90} for basic knowledge about Coxeter groups.

\subsection{Finite root systems}
Let $\Phi$ be a finite root system in a Euclidean space $(E,(\cdot,\cdot))$, $\Delta = \{ \alpha_i \mid i \in I \}$ a set of simple roots, $\Phi^+$ the corresponding set of positive roots. For $\alpha \in \Phi$, we denote its coroot by $\alpha^\vee := \frac{2}{(\alpha,\alpha)}\alpha$. For $\alpha \in \Phi$, we denote by $s_\alpha$ the reflection with respect to $\alpha$, i.e.,
$$
s_\alpha(v) := v - (v,\alpha^\vee)\alpha \qu \text{ for } v \in E.
$$

Let
$$
\Lm := \{ v \in E \mid (v,\alpha^\vee) \in \Z \Forall \alpha \in \Phi \}
$$
be the weight lattice, and
$$
\Lm^+ := \{ v \in \Lm \mid (v,\alpha^\vee) \geq 0 \Forall \alpha \in \Phi \}
$$
the set of dominant integral weights.

For $\alpha \in \Phi$ and $k \in \Z$, consider the hyperplane
$$
H_{\alpha,k} := \{ v \in E \mid (v,\alpha^\vee) = k \},
$$
and the affine reflection
$$
s_{\alpha,k} : E \rightarrow E;\ v \mapsto v - ((v,\alpha^\vee)-k)\alpha = s_\alpha(v) + k\alpha.
$$
For each $\lm \in \Lm$, let $t_{\lm} : E \rightarrow E;\ v \mapsto v + \lm$ denote the translation by $\lm$. Then, from the definition of $s_{\alpha,k}$, we see that $s_{\alpha,k} = t_{k\alpha} s_\alpha$.

Let $W$ (resp., $W_\aff$) be the subgroup of the affine transformation group of $E$ generated by $\{ s_\alpha \mid \alpha \in \Phi^+ \}$ (resp., $\{ s_{\alpha,k} \mid \alpha \in \Phi^+,\ k \in \Z \}$). It is the Weyl group (resp., affine Weyl group) associated to the coroot system $\Phi^\vee := \{ \alpha^\vee \mid \alpha \in \Phi \}$. It is well-known that $W$ (resp., $W_\aff$) is generated by $S := \{ s_i := s_{\alpha_i} \mid i \in I \}$ (resp., $S_\aff := S \sqcup \{ s_0 := s_{\theta,1}\}$, where $\theta$ is such that $\theta^\vee$ is the highest coroot). Moreover, $(W,S)$ (resp., $(W_\aff,S_\aff)$) forms a Coxeter system.

\begin{ex}\label{Type A root system}\normalfont
Let $E = \{ \sum_{i=1}^n x_i \eps_i \in \R^n \mid \sum_{i=1}^n x_i = 0  \}$, where $\{ \eps_1,\ldots,\eps_n \}$ is the standard basis of $\R^n$. Then, $\Phi := \{ \eps_i-\eps_j \mid 1 \leq i \neq j \leq n \}$ forms the root system of type $A_{n-1}$ with simple roots $\Delta := \{ \alpha_i := \eps_i-\eps_{i+1} \mid 1 \leq i \leq n-1 \}$ and positive roots $\Phi^+ := \{ \eps_i-\eps_j \mid 1 \leq i < j \leq n \}$. The Weyl group $W$ is isomorphic to the symmetric group $\frS_n$; the reflection $s_{\eps_i-\eps_j}$ corresponds to the transposition $(i,j)$. Each $\lm \in \Lm^+$ is identified with a partition $\lm = (\lm_1,\ldots,\lm_n)$ by
$$
\lm_i-\lm_{i+1} = (\lm,\alpha_i), \qu \lm_n = 0.
$$
\end{ex}

\begin{lem}\label{reflection of hyperplanes}
Let $\alpha,\beta \in \Phi$, $k,l \in \Z$. Then,
$$
s_{\beta,l}(H_{\alpha,k}) = H_{s_\beta(\alpha), k - l (\beta,\alpha^\vee)}.
$$
\end{lem}

\begin{proof}
Let $v \in H_{\alpha,k}$. Then, we have
$$
(s_{\beta,l}(v),s_\beta(\alpha)^\vee) = (s_\beta(v)+l\beta,s_\beta(\alpha^\vee)) = (v-l\beta,\alpha^\vee) = k-l(\beta,\alpha^\vee).
$$
This implies that $s_{\beta,l}(H_{\alpha,k}) \subset H_{s_\beta(\alpha),k-l(\beta,\alpha^\vee)}$. Replacing $(\alpha,k)$ by $(s_\beta(\alpha),k-l(\beta,\alpha^\vee))$, we obtain
$$
s_{\beta,l}(H_{s_\beta(\alpha),k-l(\beta,\alpha^\vee)}) \subset H_{\alpha,k-l(\beta,\alpha^\vee)-l(\beta,s_\beta(\alpha^\vee))} = H_{\alpha,k}.
$$
Hence, we conclude that $s_{\beta,l}(H_{\alpha,k}) = H_{s_\beta(\alpha),k-l(\beta,\alpha^\vee)}$.
\end{proof}

Let $w_\circ \in W$ denote the longest element. Set $N := \ell(w_\circ) = |\Phi^+|$.

\begin{defi}\normalfont
A reduced word (for $w_\circ$) is an $N$-tuple $\bfi = (i_1,\ldots,i_N) \in I^N$ such that $w_\circ = s_{i_1} \cdots s_{i_N}$.
\end{defi}

\begin{defi}\normalfont
A reflection order (also known as a convex order) is a total order $\leq$ on $\Phi^+$ satisfying the following: If $\alpha,\beta,\gamma \in \Phi^+$ is such that $\alpha < \beta$ and $\alpha + \beta = \gamma$, then we have $\alpha < \gamma < \beta$.
\end{defi}

To a reduced word $\bfi = (i_1,\ldots,i_N)$, we associate a sequence $(\beta_1,\ldots,\beta_N)$ of positive roots by
$$
\beta_1 := \alpha_{i_1}, \qu \beta_j := s_{i_1} \cdots s_{i_{j-1}}(\alpha_{i_j}) \For 2 \leq j \leq N.
$$
It is well-known that $\{ \beta_1,\ldots,\beta_N \} = \Phi^+$, and the total order $<_{\bfi}$ on $\Phi^+$ given by $\beta_1 <_{\bfi} \cdots <_{\bfi} \beta_N$ is a reflection order. Moreover, this assignment gives a bijection between the set of reduced words and the set of reflection orders.

\begin{ex}\normalfont
Consider the case when our root system is of type $A_2$. There are only two reduced words
$$
\bfi_1 := (1,2,1), \AND \bfi_2 := (2,1,2),
$$
and only two reflection orders
$$
\eps_1-\eps_2 <_{\bfi_1} \eps_1-\eps_3 <_{\bfi_1} \eps_2-\eps_3, \AND \eps_2-\eps_3 <_{\bfi_2} \eps_1-\eps_3 <_{\bfi_2} \eps_1-\eps_2.
$$
\end{ex}

\begin{rem}\normalfont\label{Remark on reflection order on coroots}
Let $\bfi = (i_1,\ldots,i_N)$ be a reduced word. Since the Weyl group for the root system $\Phi$ is the same as that for $\Phi^\vee$, the word $\bfi$ is also a reduced word for $\Phi^\vee$. Hence, the total order $\leq_{\bfi}^\vee$ on $(\Phi^\vee)^+ := \{ \alpha^\vee \mid \alpha \in \Phi^+ \}$, defined by the same way as the reflection order $\leq_\bfi$ on $\Phi^+$, is a reflection order. Note that we have $\alpha^\vee \leq_\bfi^\vee \beta^\vee$ if and only if $\alpha \leq_\bfi \beta$. Then, for each $\alpha,\beta,\gamma \in \Phi^+$ such that $\alpha <_\bfi \beta$ and $\gamma^\vee = \alpha^\vee + \beta^\vee$, we have $\alpha <_\bfi \gamma <_\bfi \beta$.
\end{rem}

\subsection{Crystals}

\begin{defi}\normalfont
A crystal is a set $B$ equipped with maps $\wt : B \rightarrow \Lm$, $\Etil_i,\Ftil_i : B \rightarrow B \sqcup \{0\}$ ($0$ is a formal symbol), $i \in I$ satisfying the following:
\begin{enumerate}
\item For each $b,b' \in B$ and $i \in I$, we have $\Ftil_i(b) = b'$ if and only if $b = \Etil_i(b')$.
\item For each $b,b' \in B$ and $i \in I$, if $\Ftil_i(b) = b'$, then $\wt(b') = \wt(b)-\alpha_i$.
\item For each $b \in B$ and $i \in I$, we have $\vphi_i(b) = \vep_i(b) + (\wt(b),\alpha_i^\vee)$, where
$$
\vphi_i(b) := \max\{ k \geq 0 \mid \Ftil_i^k(b) \neq 0 \}, \qu \vep_i(b) := \max\{ k \geq 0 \mid \Etil_i^k(b) \neq 0 \}.
$$
\end{enumerate}
\end{defi}

\begin{defi}\normalfont
Let $B_1,B_2$ be crystals. A morphism $\psi : B_1 \rightarrow B_2$ of crystals is a map $\psi : B_1 \sqcup \{0\} \rightarrow B_2 \sqcup \{0\}$ satisfying the following:
\begin{enumerate}
\item For each $b \in B_1$ and $i \in I$, if $\psi(b) \in B_2$, then we have $\wt(\psi(b)) = \wt(b)$, $\vphi_i(\psi(b)) = \vphi_i(b)$, and $\vep_i(\psi(b)) = \vep_i(b)$.
\item $\psi(0) = 0$.
\item For each $b \in B_1$ and $i \in I$, we have $\Ftil_i(\psi(b)) = \psi(\Ftil_i(b))$ and $\Etil_i(\psi(b)) = \psi(\Etil_i(b))$.
\end{enumerate}
A morphism $\psi$ is said to be an isomorphism if it is a bijection and if $\psi\inv$ is a morphism of crystals.
\end{defi}

In some literature, what we just defined are called seminormal crystals or semiregular crystals, and strict morphisms of crystals.

To a crystal $B$, we associate a colored directed graph as follows. The vertex set is $B$. For $b,b' \in B$, we put an arrow colored by $i \in I$ from $b$ to $b'$ if $b' = \Ftil_i(b)$. This graph is called the crystal graph of $B$.

The notion of crystals originates in the representation theory of complex semisimple Lie algebras (or, associated quantum groups). Given a finite-dimensional representation of the complex semisimple Lie algebra whose root system is isomorphic to our root system $\Phi$, one can obtain a crystal by extracting some information about its module structure. In particular, for each $\lm \in \Lm^+$, there exists a unique crystal $\clB(\lm)$ coming from the irreducible highest weight module $V(\lm)$ of highest weight $\lm$. One feature of $\clB(\lm)$ is the existence of the highest weight vector $b_\lm \in \clB(\lm)$; it satisfies $\wt(b_\lm) = \lm$, and
$$
\clB(\lm) = \{ \Ftil_{i_1} \cdots \Ftil_{i_l}(b_\lm) \mid l \geq 0,\ i_1,\ldots,i_l \in I \} \setminus \{0\}.
$$

\begin{ex}\normalfont
Suppose that our root system is of type $A_2$, and $\lm = (2,1,0)$. Then, the crystal graph of $\clB(\lm)$ is as follows:
$$
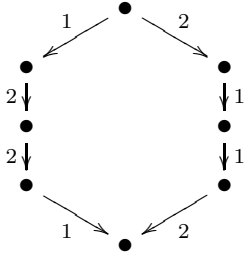
\xymatrix@R=10pt{
                      & \bullet \ar[dl]_1 \ar[dr]^2 & \\
\bullet \ar[d]_2 &                                     & \bullet \ar[d]^1 \\
\bullet \ar[d]_2 &                                     & \bullet \ar[d]^1 \\
\bullet \ar[dr]_1 &                                    & \bullet \ar[dl]^2 \\
                       &  \bullet                         &
}
$$
\end{ex}

Let $B$ be a crystal. For each $b \in B$ and $i \in I$, set
$$
\Etil_i^{\max}(b) := \Etil_i^{\vep_i(b)}(b).
$$

\begin{defi}\normalfont
Let $B$ be a crystal, $\bfi = (i_1,\ldots,i_N) \in I^N$ a reduced word. The $\bfi$-string datum $\str_{\bfi}(b)$ of $b \in B$ is an $N$-tuple of nonnegative integers given by
$$
\str_{\bfi}(b) := (\vep_{i_1}(b),\vep_{i_2}(\Etil_{i_1}^{\max}(b)),\ldots,\vep_{i_N}(\Etil_{i_{N-1}}^{\max} \cdots \Etil_{i_1}^{\max}(b))).
$$
\end{defi}

When $B \simeq \clB(\lm)$, it is known that the map $B \rightarrow \Z_{\geq 0}^N;\ b \mapsto \str_{\bfi}(b)$ is injective since $\Etil_{i_{N}}^{\max} \cdots \Etil_{i_1}^{\max}(b) = b_\lm$ for all $b \in \clB(\lm)$.

\begin{lem}\label{sufficient condition on isomorphisms}
Let $B_1,B_2$ be crystals isomorphic to $\clB(\lm)$, $\bfi$ a reduced word. Suppose that there exists a bijection $\psi : B_1 \rightarrow B_2$ such that $\str_\bfi(\psi(b)) = \str_\bfi(b)$ for all $b \in B_1$. Then, $\psi$ is an isomorphism of crystals.
\end{lem}

\begin{proof}
Without loss of generality, we may assume that $B_1 = B_2 = \clB(\lm)$. Then, the injectivity of $\str_\bfi : \clB(\lm) \rightarrow \Z_{\geq 0}^N$ implies that $\psi(b) = b$ for all $b \in \clB(\lm)$. Hence, $\psi$ is the identity map on $\clB(\lm)$, which is an isomorphism of crystals. This completes the proof.
\end{proof}

\subsection{Gelfand-Tsetlin patterns}\label{GT patterns}
In this subsection, assume that our root system is of type $A_{n-1}$. In particular, we identify the dominant integral weight $\lm$ with the partition $(\lm_1,\ldots,\lm_n)$ as in Example \ref{Type A root system}.

\begin{defi}\normalfont
A Gelfand-Tsetlin pattern of shape $\lm$ is a tuple $\bfa = (a_{i,j})_{1 \leq i \leq j \leq n}$ of nonnegative integers satisfying the following:
\begin{itemize}
\item $a_{i,i} = \lm_i$ for all $i = 1,\ldots,n$.
\item $a_{i+1,j} \leq a_{i,j} \leq a_{i,j-1}$ for all $1 \leq i < j \leq n$.
\end{itemize}
Let $\GT(\lm)$ denote the set of Gelfand-Tsetlin patterns of shape $\lm$.
\end{defi}

$\GT(\lm)$ is in a natural bijection with the set $\SST(\lm)$ of semistandard Young tableaux of shape $\lm$ filled with letters in $\{ 1,\ldots,n \}$. The bijection is given as follows. Let $T \in \SST(\lm)$. We denote by $T(i,j)$ the entry of the box in the $i$-th row and the $j$-th column. Then, the corresponding Gelfand-Tsetlin pattern $\bfa = (a_{i,j})_{1 \leq i \leq j \leq n}$ is given as follows; for each $j = 1,\ldots,n$, the tuple $(a_{1,j},a_{2,j+1},\ldots,a_{n-j+1,n})$ is the partition representing the shape of the tableau obtained from $T$ by deleting the boxes whose entries are greater than $n-j+1$. In other words,
$$
a_{i,j} = \sharp \{ (k,l) \mid k = i \AND T(k,l) \leq n-j+i \}.
$$
Via this bijection, $\GT(\lm)$ is equipped with a crystal structure isomorphic to $\clB(\lm)$.

\begin{ex}\label{Example 232}\normalfont
Suppose that $n = 3$, and $\lm = (2,1,0)$. The following are the crystal graphs of $\SST(\lm)$ and $\GT(\lm)$:
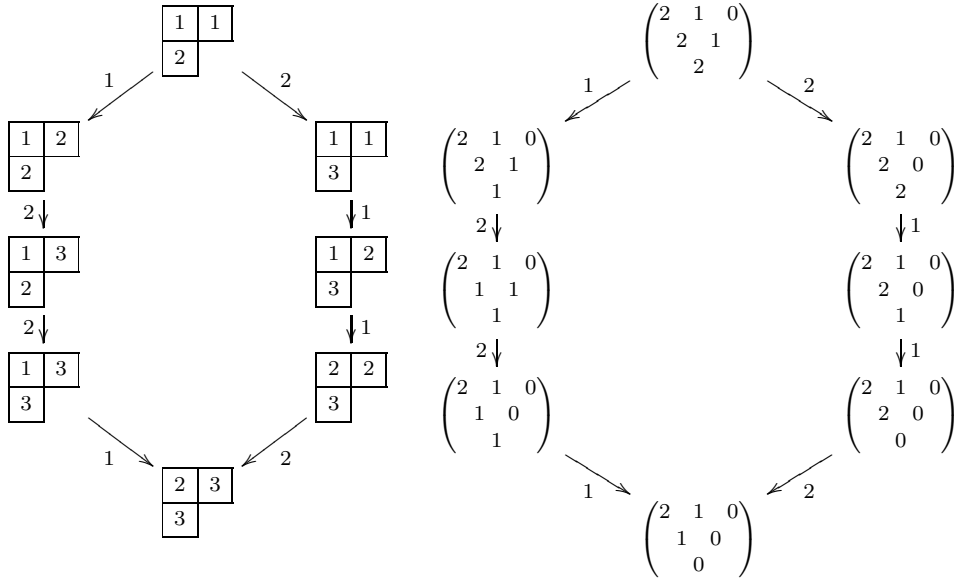
$$
\xymatrix@R=10pt{
                      & \text{\tiny{$\ytableausetup{centertableaux} \begin{ytableau} 1 & 1 \\ 2 \end{ytableau}$}} \ar[dl]_1 \ar[dr]^2 & \\
\text{\tiny{$\ytableausetup{centertableaux} \begin{ytableau} 1 & 2 \\ 2 \end{ytableau}$}} \ar[d]_2 &                                     & \text{\tiny{$\ytableausetup{centertableaux} \begin{ytableau} 1 & 1 \\ 3 \end{ytableau}$}} \ar[d]^1 \\
\text{\tiny{$\ytableausetup{centertableaux} \begin{ytableau} 1 & 3 \\ 2 \end{ytableau}$}} \ar[d]_2 &                                     & \text{\tiny{$\ytableausetup{centertableaux} \begin{ytableau} 1 & 2 \\ 3 \end{ytableau}$}} \ar[d]^1 \\
\text{\tiny{$\ytableausetup{centertableaux} \begin{ytableau} 1 & 3 \\ 3 \end{ytableau}$}} \ar[dr]_1 &                                    & \text{\tiny{$\ytableausetup{centertableaux} \begin{ytableau} 2 & 2 \\ 3 \end{ytableau}$}} \ar[dl]^2 \\
                       &  \text{\tiny{$\ytableausetup{centertableaux} \begin{ytableau} 2 & 3 \\ 3 \end{ytableau}$}}                       &
} \qu 
\xymatrix@R=10pt{
                      & \text{{\tiny $\begin{pmatrix} 2 \qu 1 \qu 0 \\ 2 \qu 1 \\ 2  \end{pmatrix}$}} \ar[dl]_1 \ar[dr]^2 & \\
\text{{\tiny $\begin{pmatrix} 2 \qu 1 \qu 0 \\ 2 \qu 1 \\ 1 \end{pmatrix}$}} \ar[d]_2 &                                     & \text{{\tiny $\begin{pmatrix} 2 \qu 1 \qu 0 \\ 2 \qu 0 \\ 2  \end{pmatrix}$}} \ar[d]^1 \\
\text{{\tiny $\begin{pmatrix} 2 \qu 1 \qu 0 \\ 1 \qu 1 \\ 1  \end{pmatrix}$}} \ar[d]_2 &                                     & \text{{\tiny $\begin{pmatrix} 2 \qu 1 \qu 0 \\ 2 \qu 0 \\ 1  \end{pmatrix}$}} \ar[d]^1 \\
\text{{\tiny $\begin{pmatrix} 2 \qu 1 \qu 0 \\ 1 \qu 0 \\ 1  \end{pmatrix}$}} \ar[dr]_1 &                                    & \text{{\tiny $\begin{pmatrix} 2 \qu 1 \qu 0 \\ 2 \qu 0 \\ 0  \end{pmatrix}$}} \ar[dl]^2 \\
                       &  \text{{\tiny $\begin{pmatrix} 2 \qu 1 \qu 0 \\ 1 \qu 0 \\ 0  \end{pmatrix}$}}                        &
}
$$
here, we display $(a_{i,j})_{1 \leq i \leq j \leq 3} \in \GT(\lm)$ as $\begin{pmatrix} a_{11} \qu a_{22} \qu a_{33} \\ a_{12} \qu a_{23} \\ a_{13} \end{pmatrix}$.
\end{ex}

Let us consider the following $N$-tuple:
$$
\bfi_A := (1,2,1,3,2,1,\ldots,n-1,n-2,\ldots,1) \in I^N.
$$
As is well-known, this is a reduced word. Let us write $\Phi^+ = \{ \gamma_1, \ldots, \gamma_N \}$ in a way such that $\gamma_1 <_{\bfi_A} \cdots <_{\bfi_A} \gamma_N$. Also, for each $1 \leq i < j \leq n$ and $1 \leq k < l \leq n$, we write $(i,j) <_{\bfi_A} (k,l)$ if $\eps_i-\eps_j <_{\bfi_A} \eps_k-\eps_l$. Explicitly, we have $(i,j) <_{\bfi_A} (k,l)$ if and only if either $(1)$ $j < l$ or $(2)$ $j = l$ and $i < k$. For example, we have
$$
(1,2) <_{\bfi_A} (1,3) <_{\bfi_A} (2,3) <_{\bfi_A} (1,4) <_{\bfi_A} (2,4) <_{\bfi_A}  (3,4) <_{\bfi_A} (1,5) <_{\bfi_A} \cdots.
$$

In the sequel, we consider the $\bfi_A$-string datum of various crystals. Let $B$ be a crystal, $b \in B$. Let us write $\str_{\bfi_A}(b) = (d_1,\ldots,d_N)$. It is convenient to write $\str_{\bfi_A}(b) = (d_{i,j})_{1 \leq i < j \leq n}$, where $d_{i,j} = d_k$ if $\gamma_k = \eps_i-\eps_j$.

\begin{prop}\label{String datum of GT}
Let $\bfa = (a_{i,j})_{1 \leq i \leq j \leq n} \in \GT(\lm)$. Then, the $\bfi_A$-string datum of $\bfa$ is given by $\str_{\bfi_A}(\bfa) = (d_{i,j}(\bfa))_{1 \leq i < j \leq n}$, where
$$
d_{i,j}(\bfa) = \sum_{m=1}^{j-i} (a_{m,m+n-j} - a_{m,m+n-j+1}).
$$
\end{prop}

\begin{proof}
Let $T$ denote the semistandard tableau of shape $\lm$ corresponding to $\bfa$. Recall that we have
$$
a_{i,j} = \sharp \{ (k,l) \mid k = i \AND T(k,l) \leq n-j+i \}.
$$
Then, we see that
$$
a_{m,m+n-j} - a_{m,m+n-j+1} = \sharp \{ (k,l) \mid k = m \AND T(k,l) = j \},
$$
and hence,
$$
\sum_{m=1}^{j-i} (a_{m,m+n-j} - a_{m,m+n-j+1}) = \sharp \{ (k,l) \mid 1 \leq k \leq j-i \AND T(k,l) = j \}.
$$
By \cite[Proposition 11.2 (1)]{BS17}, this is nothing but the $(i,j)$-th entry of $\str_{\bfi_A}(T)$. Thus, the assertion follows.
\end{proof}

\section{Alcove paths model}
In this section, we review basic results from \cite{LP07}, \cite{LP08}, \cite{L07} concerning the alcove path model.

\subsection{Alcove paths and admissible subsets}
\begin{defi}\normalfont
An alcove is a connected component of $E \setminus \bigcup_{\substack{\alpha \in \Phi^+ \\ k \in \Z}} H_{\alpha,k}$. The fundamental alcove $A_\circ$ is the alcove defined by
$$
A_\circ = \{ v \in E \mid 0 < (v,\alpha^\vee) < 1 \Forall \alpha \in \Phi^+ \}.
$$
\end{defi}

Two alcoves $A,B$ are said to be adjacent if $A \neq B$ and if their closures have a common facet (face of codimension $1$). Such a common facet $F$ is unique. In general, if $F$ is a facet of an alcove, then there exist unique $\beta \in \Phi^+$ and $l \in \Z$ such that $F \subset H_{\beta,l}$. In this case, we set $s_F := s_{\beta,l}$.

\begin{defi}\normalfont
An alcove path is a sequence $\Pi := (A_0,A_1,\ldots,A_s)$ of alcoves such that $A_{i-1}$ and $A_i$ are adjacent for all $i = 1,\ldots,s$. An alcove path is said to be reduced if it has minimal length among all alcove paths from $A_0$ to $A_s$. The sequence of positive roots associated to $\Pi$ is $\Gamma(\Pi) = (\beta_1,\ldots,\beta_s)$, where $\beta_i$ is the positive root such that the common facet $F_i$ of $A_{i-1}$ and $A_i$ lies in $H_{\beta_i,l_i}$ for some $l_i \in \Z$.
\end{defi}

\begin{lem}[{\cite[Lemma 5.3]{LP07}}]\label{reduced alcove path and reduced expression}
Let $v \in W_\aff$. Then, there exists a one-to-one correspondence between the set of reduced expressions of $v$ and the set of reduced alcove paths from $A_\circ$ to $vA_\circ$.
\end{lem}

From now on, we fix $\lm \in \Lm^+$. Let $\AP(\lm)$ denote the set of reduced alcove paths from $A_\circ$ to $A_\circ - \lm$. Let $\Pi = (A_0,\ldots,A_s) \in \AP(\lm)$ with $\Gamma(\Pi) = (\beta_1,\ldots,\beta_s)$.

\begin{lem}
Let $i \in \{ 1,\ldots,s \}$. Then, we have
$$
l_i = - \sharp \{ j < i \mid \beta_j = \beta_i \}.
$$
\end{lem}

\begin{proof}
Since $\Pi$ is reduced, a hyperplane of the form $H_{\beta_i,k}$ lies between $A_0 = A_\circ$ and $A_i$ if and only if $l_i \leq k \leq 0$. Then, for each such $k$, there is a unique $j \leq i$ such that $\beta_j = \beta_i$ and $l_j = k$. Thus, the assertion follows.
\end{proof}

The sequence of positive roots $\Gamma(\Pi)$ is characterized by the following conditions:

\begin{prop}[{\cite[Proposition 4.4]{LP08}}]\label{Characterization of chain}
Let $\beta_1,\ldots,\beta_s \in \Phi^+$. Then, there exists $\Pi \in \AP(\lm)$ such that $\Gamma(\Pi) = (\beta_1,\ldots,\beta_s)$ if and only if the following two conditions are satisfied:
\begin{itemize}
\item For each $\beta \in \Phi^+$, we have $\sharp \{ i \mid \beta_i = \beta \} = (\lm,\beta^\vee)$,
\item For each $\alpha,\beta,\gamma \in \Phi^+$ such that $\gamma^\vee = \alpha^\vee + \beta^\vee$, consider the subsequence $(\beta_{i_1},\ldots,\beta_{i_k})$ of $(\beta_1,\ldots,\beta_s)$ consisting of $\alpha,\beta,\gamma$. Then, $\beta_{i_m} \in \{ \alpha,\beta \}$ if $m$ is odd, while $\beta_{i_m} = \gamma$ if $m$ is even.
\end{itemize}
\end{prop}

\begin{rem}\normalfont\label{Remark on betaik 1}
In the second condition, we have $k = (\lm,\alpha^\vee) + (\lm,\beta^\vee) + (\lm,\gamma^\vee) = 2(\lm,\gamma^\vee)$. In particular, $\beta_{i_k} = \gamma$.
\end{rem}

\begin{defi}\normalfont
Let $\Pi \in \AP(\lm)$ with $\Gamma(\Pi) = (\beta_1,\ldots,\beta_s)$. An admissible subset associated to $\Pi$ is a subset $J = \{ j_1,\ldots, j_t \}$ of $\{ 1,\ldots,s \}$ such that $j_1 < \cdots < j_t$ and that there exists a saturated chain
$$
e \rightarrow s_{\beta_{j_1}} \rightarrow s_{\beta_{j_1}} s_{\beta_{j_2}} \rightarrow \cdots \rightarrow s_{\beta_{j_1}} s_{\beta_{j_2}} \cdots s_{\beta_{j_t}}
$$
in the Bruhat graph of $W$, i.e., for each $k \in \{ 1,\ldots,t \}$, we have $\ell(s_{\beta_{j_1}} \cdots s_{\beta_{j_k}}) = k$. We understand that the empty set is an admissible subset. Let $\clA(\Pi)$ denote the set of admissible subsets associated to $\Pi$.
\end{defi}

\begin{rem}\normalfont
When we consider a subset $J = \{ j_1,\ldots,j_t \} \subset \{ 1,\ldots,s \}$, we always assume that $j_1 < \cdots < j_t$.
\end{rem}

\begin{ex}\normalfont
Suppose that our root system is of type $A_2$, and $\lm = (2,1,0)$. Let
\begin{align}
\begin{split}
&\Pi_1 = (A_\circ,s_1 A_\circ, s_1s_2 A_\circ,s_1s_2s_1 A_\circ, s_1s_2s_1s_0A_\circ), \\
&\Pi_2 = (A_\circ, s_2 A_\circ, s_2s_1 A_\circ, s_2s_1s_2 A_\circ, s_2s_1s_2s_0A_\circ).
\end{split} \nonumber
\end{align}
Then, these are elements of $\AP(\lm)$. We have
\begin{align}
\begin{split}
&\Gamma(\Pi_1) = (\alpha_1,\alpha_1+\alpha_2,\alpha_2,\alpha_1+\alpha_2), \\
&\Gamma(\Pi_2) = (\alpha_2,\alpha_1+\alpha_2,\alpha_1,\alpha_1+\alpha_2).
\end{split} \nonumber
\end{align}
The admissible subsets are the following:
\begin{align}
\begin{split}
&\clA(\Pi_1) = \{ \emptyset, \{1\}, \{3\}, \{1,2\}, \{1,3\}, \{1,4\}, \{3,4\}, \{1,2,3\} \}, \\
&\clA(\Pi_2) = \{ \emptyset, \{3\}, \{1\}, \{1,3\}, \{1,2\}, \{3,4\}, \{1,4\}, \{1,2,3\} \}.
\end{split} \nonumber
\end{align}
\end{ex}

Let $\Pi = (A_0,\ldots,A_s) \in \AP(\lm)$, $J = \{ j_1, \ldots, j_t \} \in \clA(\Pi)$. For $1 \leq k < l \leq t$, set
$$
w_{k,l}(J) := s_{F_{j_k}} s_{F_{j_{k+1}}} \cdots s_{F_{j_l}} \in W_{\aff}, \qu \ol{w}_{k,l}(J) := s_{\beta_{j_k}} s_{\beta_{j_{k+1}}} \cdots s_{\beta_{j_l}} \in W.
$$
When $(k,l) = (1,t)$, we abbreviate $w_{1,t}(J)$ and $\ol{w}_{1,t}(J)$ as $w(J)$ and $\ol{w}(J)$, respectively.

\subsection{Galleries}
For our purposes, it is convenient to rewrite the admissible subsets in terms of galleries, which we recall now.

\begin{defi}\normalfont
A gallery is a sequence $\gamma = (A_0,F_1,A_1,F_2,A_2,\ldots,F_s,A_s,\mu)$ satisfying the following:
\begin{itemize}
\item $A_0,A_1,\ldots,A_s$ are alcoves.
\item $F_i$ is a common facet of $A_{i-1}$ and $A_i$.
\item $\mu \in \Lm$ is a vertex of (the closure of) $A_s$.
\end{itemize}
Given a gallery $\gamma$, set
$$
J(\gamma) := \{ i \mid A_{i-1} = A_i \}.
$$
\end{defi}

\begin{ex}\normalfont
Let $\Pi = (A_0,\ldots,A_s)$ be an alcove path. Let $F_i$ denote the unique common facet of $A_{i-1}$ and $A_i$. Then, for each vertex $\mu \in \Lm$ of $A_s$, the sequence $\gamma(\Pi;\mu) := (A_0,F_1,A_1,\ldots,F_s,A_s,\mu)$ is a gallery. In this case, we have $J(\gamma(\Pi;\mu)) = \emptyset$.
\end{ex}

\begin{defi}\normalfont
Let $\gamma = (A_0,F_1,A_1,\ldots,F_s,A_s,\mu)$ be a gallery, and $j \in \{ 1,\ldots,s \}$.
\begin{enumerate}
\item Let $\phi_j(\gamma) = (A'_0,F'_1,A'_1,\ldots,F'_s,A'_s,\mu')$ be the gallery defined by
\begin{enumerate}
\item $A'_i := \begin{cases}
A_i \qu & \IF 0 \leq i < j, \\
s_{F_j}(A_i) \qu & \IF j \leq i \leq s.
\end{cases}$
\item $F'_i := \begin{cases}
F_i \qu & \IF 1 \leq i < j, \\
s_{F_j}(F_i) \qu & \IF j \leq i \leq s.
\end{cases}$
\item $\mu' := s_{F_j}(\mu)$.
\end{enumerate}
\item For a subset $J = \{ j_1, \ldots, j_t \} \subset \{ 1,\ldots,s \}$ with $j_1 < \cdots < j_t$, set
$$
\phi_J(\gamma) := \phi_{j_1} \cdots \phi_{j_t}(\gamma).
$$
\end{enumerate}
Here, we understand that $\phi_J(\gamma) = \gamma$ if $J = \emptyset$.
\end{defi}

For each alcove path $\Pi = (A_0,A_1,\ldots,A_s)$ and a vertex $\mu$ of $A_s$, let $G(\Pi;\mu)$ denote the set of galleries of the form $\phi_J(\gamma(\Pi;\mu))$, $J \subset \{ 1,\ldots,s \}$. Since $J(\phi_J(\gamma(\Pi;\mu))) = J$, we may identify $\phi_J(\gamma(\Pi;\mu))$ with $J$.

\subsection{Crystal structure}
Throughout this subsection, we fix an alcove path $\Pi = (A_0,A_1,\ldots,A_s)$ and a vertex $\mu$ of $A_s$. We identify $\phi_J(\gamma(\Pi;\mu)) \in G(\Pi;\mu)$ with $J \subset \{ 1,\ldots,s \}$. Now, we define maps $\wt : G(\Pi;\mu) \rightarrow \Lm$ and $\Etil_p,\Ftil_p : G(\Pi;\mu) \rightarrow G(\Pi;\mu) \sqcup \{0\}$ (here, $0$ is a formal symbol) for each $p \in I$. Let $J \in G(\Pi;\mu)$ and $p \in I$. Let us write $J = (A^J_0,F^J_1,A^J_1,\ldots,F^J_s,A^J_s,\mu^J)$. First, we define $\wt : G(\Pi;\mu) \rightarrow \Lm$ by
$$
\wt(J) := -\mu^J.
$$

Let $\beta^J_i \in \Phi^+$ and $l^J_i \in \Z$ be such that $F^J_i \subset H_{\beta^J_i,l^J_i}$. Set
\begin{itemize}
\item $I(J,p) := \{ i \in \{ 1,\ldots,s \} \mid \beta^J_i = \alpha_p \}$.
\item $L(J,p) := \{ l^J_i \}_{i \in I(J,p)} \cup \{ (\mu^J,\alpha_p^\vee) \}$.
\item $M(J,p) := \min L(J,p)$.
\end{itemize}
We define $\Ftil_p(J) \in G(\Pi;\mu) \sqcup \{0\}$ by
$$
\Ftil_p(J) := \begin{cases}
0 \qu & \IF M(J,p) \geq 0, \\
(J \setminus \{m_F\}) \cup \{k_F\} \qu & \IF M(J,p) < 0 \AND \{ i \in I(J,p) \mid l^J_i = M(J,p) \} \neq \emptyset, \\
J \sqcup \{k'\} \qu & \IF M(J,p) < 0 \AND \{ i \in I(J,p) \mid l^J_i = M(J,p) \} = \emptyset,
\end{cases}
$$
where
\begin{align}
\begin{split}
m_F &:= \min \{ i \in I(J,p) \mid l^J_i = M(J,p) \}, \\
k_F &:= \max(I(J,p) \cap \{ 1,\ldots,m_F-1 \}), \\
k' &:= \max I(J,p).
\end{split} \nonumber
\end{align}

Also, we define $\Etil_p(J) \in G(\Pi;\mu) \sqcup \{0\}$ by
$$
\Etil_p(J) := \begin{cases}
0 \qu & \IF M(J,p) = (\mu^J,\alpha_p^\vee), \\
(J \setminus \{k_E\}) \cup \{m_E\} \qu & \IF M(J,p) < (\mu^J,\alpha_p^\vee) \AND k_E \neq k', \\
J \setminus \{k'\} \qu & \IF M(J,p) < (\mu^J,\alpha_p^\vee) \AND k_E = k',
\end{cases}
$$
where
\begin{align}
\begin{split}
k_E &:= \max \{ i \in I(J,p) \mid l^J_i = M(J,p) \}, \\
m_E &:= \min(I(J,p) \cap \{ k_E+1,\ldots,s \}).
\end{split} \nonumber
\end{align}
Note that by definition, it always holds that $M(J,p) \leq (\mu^J,\alpha_p^\vee)$.

\begin{rem}\normalfont
When $\Pi \in \AP(\lm)$, $J \in \clA(\Pi)$, and $\mu = -\lm$, the maps $\wt,\Etil_p,\Ftil_p$ just defined above are the same as those defined in \cite[Section 3.7]{L07} (note that our $l^J_i$ and $M(J,p)$ are the negative of those in \cite{L07}).
\end{rem}

\begin{theo}[{\cite[Corollary 4.9]{L07}}]
Let $\lm \in \Lm^+$, and $\Pi \in \AP(\lm)$. Then $\clA(\Pi)$, regarded as a subset of $G(\Pi;-\lm)$, is closed under $\Etil_p,\Ftil_p$, $p \in I$. Namely, for each $J \in \clA(\Pi)$, we have $\Etil_p(J),\Ftil_p(J) \in \clA(\Pi) \sqcup \{0\}$. Moreover, $\clA(\Pi)$ equipped with the maps $\wt,\Etil_p,\Ftil_p$, $p \in I$ is a crystal isomorphic to $\clB(\lm)$ in a way such that $\emptyset \in \clA(\Pi)$ corresponds to $b_\lm \in \clB(\lm)$.
\end{theo}

Now, we collect basic properties of the crystal structure of $\clA(\Pi)$ which are needed for later argument; see \cite{L07} and \cite{LP08} for details.

\begin{prop}\label{Properties of alcove paths model}
Let $\Pi = (A_0,\ldots,A_s) \in \AP(\lm)$, $J = \{ j_1,\ldots, j_t \} \in \clA(\Pi)$, $p \in I$. Then, the following hold:
\begin{enumerate}
\item\label{Properties of alcove paths model 1} $\wt(J) = -w(J)(-\lm)$.
\item\label{Properties of alcove paths model 2} $M(J,p) \leq 0$.
\item\label{Properties of alcove paths model 3} $M(J,p) = \min(\{ l^J_i \mid i \in \Itil(J,p) \} \cup \{ (-\wt(J),\alpha_p^\vee) \})$, where $\Itil(J,p) := I(J,p) \cap J$.
\item\label{Properties of alcove paths model 4} $\vphi_p(J) = -M(J,p)$.
\item\label{Properties of alcove paths model 5} $\vep_p(J) = (-\wt(J),\alpha_p^\vee) - M(J,p)$.
\item\label{Properties of alcove paths model 6} If $M(J,p) < 0 \AND \{ i \in I(J,p) \mid l^J_i = M(J,p) \} \neq \emptyset$, then $m_F \in J$, $k_F \notin J$, and $\ol{w}(\Ftil_p(J)) = \ol{w}(J)$.
\item\label{Properties of alcove paths model 7} If $M(J,p) < 0 \AND \{ i \in I(J,p) \mid l^J_i = M(J,p) \} = \emptyset$, then $k' \notin J$ and $\ol{w}(\Ftil_p(J)) = s_p\ol{w}(J) > \ol{w}(J)$.
\item\label{Properties of alcove paths model 8} If $M(J,p) < (\mu^J,\alpha_p^\vee) \AND k_E \neq k'$, then $k_E \in J$, $m_E \notin J$, and $\ol{w}(\Etil_p(J)) = \ol{w}(J)$.
\item\label{Properties of alcove paths model 9} If $M(J,p) < (\mu^J,\alpha_p^\vee) \AND k_E = k'$, then $k' \in J$ and $\ol{w}(\Etil_p(J)) = s_p\ol{w}(J) < \ol{w}(J)$.
\end{enumerate}
\end{prop}

\subsection{Yang-Baxter moves}\label{subsection, Yang-Baxter move}
Let $\Pi_1,\Pi_2 \in \AP(\lm)$. As we have seen above, both $\clA(\Pi_1)$ and $\clA(\Pi_2)$ are equipped with crystal structures isomorphic to $\clB(\lm)$. In particular, there exists a unique isomorphism $\clA(\Pi_1) \rightarrow \clA(\Pi_2)$ of crystals. Such an isomorphism can be realized as a sequence of Yang-Baxter moves, which we briefly explain now.

Recall from Lemma \ref{reduced alcove path and reduced expression} that each $\Pi \in \AP(\lm)$ corresponds to a reduced expression of $v_\lm$, where $v_\lm \in W_{\aff}$ is such that $v_\lm A_\circ = A_\circ - \lm$. By Matsumoto's theorem, any two reduced expressions of $v_\lm$ can be transformed from one into the other by a sequence of braid moves. The Yang-Baxter moves are the translations of the braid moves in the language of alcove paths.

A sequence of Yang-Baxter moves which transforms $\Pi_1$ into $\Pi_2$ induces an isomorphism $Y : \clA(\Pi_1) \rightarrow \clA(\Pi_2)$ of crystals. For the precise definition of this isomorphism, see \cite[Section 4]{L07}.

\section{Extended alcove path model}
\subsection{Extended alcove path model}\label{subsection, Extended alcove paths model}
In this subsection, we introduce the notion of extended alcove path model, which also gives a combinatorial realization of the highest weight crystals.

Let $\widetilde{\AP}(\lm)$ denote the set of reduced alcove paths from $A_\circ$ to $w_\circ A_\circ - \lm$. Let $\Pi = (A_0,\ldots,A_u) \in \widetilde{\AP}(\lm)$ with $\Gamma(\Pi) = (\beta_1,\ldots,\beta_u)$. Let $l_i \in \Z$ be such that the common facet of $A_{i-1}$ and $A_i$ is contained in the hyperplane $H_{\beta_i,l_i}$. As in the ordinary alcove path model, we have
\begin{align}\label{Level function for extended alcove path}
l_i = - \sharp \{ j < i \mid \beta_j = \beta_i \}
\end{align}
for all $i = 1,\ldots,u$.

By the arguments in \cite[Propositions 10.2--10.3]{LP08}, the sequences $\Gamma(\Pi)$ of positive roots associated to reduced alcove paths $\Pi \in \widetilde{\AP}(\lm)$ are characterized as follows (compare with Proposition \ref{Characterization of chain}).

\begin{prop}\label{Characterization of extended chain}
Let $\beta_1,\ldots,\beta_u \in \Phi^+$. Then, there exists $\Pi \in \widetilde{\AP}(\lm)$ such that $\Gamma(\Pi) = (\beta_1,\ldots,\beta_u)$ if and only if the following two conditions are satisfied:
\begin{itemize}
\item For each $\beta \in \Phi^+$, we have $\sharp \{ i \mid \beta_i = \beta \} = (\lm,\beta^\vee) + 1$.
\item For each $\alpha,\beta,\gamma \in \Phi^+$ such that $\gamma^\vee = \alpha^\vee + \beta^\vee$, consider the subsequence $(\beta_{i_1},\ldots,\beta_{i_k})$ of $(\beta_1,\ldots,\beta_u)$ consisting of $\alpha,\beta,\gamma$. Then, $\beta_{i_m} \in \{ \alpha,\beta \}$ if $m$ is odd, while $\beta_{i_m} = \gamma$ if $m$ is even.
\end{itemize}
\end{prop}

\begin{rem}\normalfont\label{Remark on betaik 2}
In the second condition, we have $k = ((\lm,\alpha^\vee)+1) + ((\lm,\beta^\vee)+1) + ((\lm,\gamma^\vee)+1) = 2(\lm,\gamma^\vee)+3$. In particular, $\beta_{i_k} \in \{ \alpha,\beta \}$.
\end{rem}

\begin{defi}\normalfont
An admissible subset associated to $\Pi = (A_0,\ldots,A_u) \in \widetilde{\AP}(\lm)$ is a subset $J = \{ j_1, \ldots, j_N \}$ of $\{ 1,\ldots,u \}$ such that $j_1 < \cdots < j_N$ and that there exists a saturated chain
$$
e \rightarrow s_{\beta_{j_1}} \rightarrow s_{\beta_{j_1}} s_{\beta_{j_2}} \rightarrow \cdots \rightarrow s_{\beta_{j_1}} s_{\beta_{j_2}} \cdots s_{\beta_{j_N}} = w_\circ
$$
in the Bruhat graph of $W$. Let $\clA(\Pi)$ denote the set of admissible subsets associated to $\Pi$. For $J = \{ j_1, \ldots, j_N \} \in \clA(\Pi)$ and $1 \leq k < l \leq N$, we define $\ol{w}_{k,l}(J)$, $w_{k,l}(J)$, $\ol{w}(J)$, and $w(J)$ by the same way as in the ordinary alcove paths model.
\end{defi}

\begin{rem}\normalfont\label{Remark on w(J)}
Let $\Pi \in \widetilde{\AP}(\lm)$, $J \in \clA(\Pi)$. Opposed to the ordinary alcove path model, the size of $J$ and the Weyl group element $\ol{w}(J)$ is independent of $J$; we have $|J| = N$ and $\ol{w}(J) = w_\circ$. However, $w(J)$ depends on $J$; there exists $\nu = \nu(J) \in \Lm$ such that $w(J) = t_\nu w_\circ$.
\end{rem}

\begin{ex}\label{Example 415}\normalfont
Suppose that our root system is of type $A_2$, and $\lm = (2,1,0)$. Let $\Pi \in \widetilde{\AP}(\lm)$ be such that
$$
\Gamma(\Pi) = (\alpha_2,\alpha_1+\alpha_2,\alpha_2,\alpha_1+\alpha_2,\alpha_1,\alpha_1+\alpha_2,\alpha_1).
$$
The admissible subsets are the following:
$$
\clA(\Pi) = \{ \{1,2,5\}, \{1,2,7\}, \{1,4,5\}, \{1,4,7\}, \{1,6,7\}, \{3,4,5\}, \{3,4,7\}, \{3,6,7\} \}.
$$
\end{ex}

\begin{lem}\label{extendable of alcove path}
Let $\Pi = (A_0,A_1,\ldots,A_s) \in \AP(\lm)$ with $\Gamma(\Pi) = (\beta_1,\ldots,\beta_s)$. Let $\bfi = (i_1,\ldots,i_N)$ be a reduced word, and consider the corresponding reflection order $\leq_{\bfi}$. Let us write $\Phi^+  = \{ \gamma_1,\ldots,\gamma_N \}$ in a way such that $\gamma_1 <_\bfi \cdots <_\bfi \gamma_N$. Set
$$
\Gamma := (\beta_1,\ldots,\beta_s,\beta_{s+1},\ldots,\beta_{s+N}), \qu \beta_{s+i} := \gamma_i \For 1 \leq i \leq N.
$$
Then, there exists $\widetilde{\Pi} \in \widetilde{\AP}(\lm)$ such that
$$
\widetilde{\Pi} = (A_0,A_1,\ldots,A_s,A_{s+1},\ldots,A_{s+N})
$$
for some alcoves $A_{s+1},\ldots,A_{s+N}$, and that $\Gamma(\widetilde{\Pi}) = \Gamma$.
\end{lem}

\begin{proof}
By Proposition \ref{Characterization of extended chain}, it suffices to show that $\Gamma$ satisfies the conditions there. First, let $\beta \in \Phi^+$. Since $\{ \beta_{s+1},\ldots,\beta_{s+N} \} = \Phi^+$, we have
$$
\sharp\{ i \in \{ 1,\ldots,s+N \} \mid \beta_i = \beta \} = \sharp\{ i \in \{ 1,\ldots,s \} \mid \beta_i = \beta \} + 1.
$$
By Proposition \ref{Characterization of chain}, we see that
$$
\sharp\{ i \in \{ 1,\ldots,s \} \mid \beta_i = \beta \} = (\lm,\beta^\vee).
$$
Therefore, we obtain the first condition.

Next, let $\alpha,\beta,\gamma \in \Phi^+$ be such that $\gamma^\vee = \alpha^\vee + \beta^\vee$, and consider the subsequence $(\beta_{i_1},\ldots,\beta_{i_k})$ of $(\beta_1,\ldots,\beta_{s+N})$ consisting of $\alpha,\beta,\gamma$. Since $\{ \beta_{s+1},\ldots,\beta_{s+N} \} = \Phi^+$, only the last three terms are of the form $\beta_{s+i}$, $i = 1,\ldots,N$. By remark \ref{Remark on reflection order on coroots}, $(\beta_{i_{k-2}}, \beta_{i_{k-1}}, \beta_{i_k})$ is either $(\alpha,\gamma,\beta)$ or $(\beta,\gamma,\alpha)$. On the other hand, by Proposition \ref{Characterization of chain}, for each $m = 1,\ldots,k-3$, we have $\beta_{i_m} \in \{ \alpha,\beta \}$ if $m$ is odd, and $\beta_{i_m} = \gamma$ if $m$ is even. Also, by remark \ref{Remark on betaik 1}, we see that $k-2$ is even. By above, we conclude that for each $m = 1,\ldots,k$, we have $\beta_{i_m} \in \{ \alpha,\beta \}$ if $m$ is odd, and $\beta_{i_m} = \gamma$ if $m$ is even. This proves the second condition. Thus, the proof completes.
\end{proof}

\begin{lem}\label{added facet contains lm}
Let $\Pi = (A_0,A_1,\ldots,A_s) \in \AP(\lm)$ and $\widetilde{\Pi} = (A_0,A_1,\ldots,A_s,A_{s+1},\ldots,A_{s+N}) \in \widetilde{\AP}(\lm)$ be as before. Let $l_i \in \Z$ be such that the common facet $F_i$ of $A_{i-1}$ and $A_i$ is contained in the hyperplane $H_{\beta_i,l_i}$. Then, we have $l_{s+k} = -(\lm,\beta_{s+k}^\vee)$, and consequently, $F_{s+k}$ contains $-\lm$ for all $k = 1,\ldots,N$.
\end{lem}

\begin{proof}
By equation \eqref{Level function for extended alcove path}, we have
$$
l_{s+k} = -\sharp\{ i \in \{ 1,\ldots,s \} \mid \beta_i = \beta_{s+k} \} = -(\lm,\beta_{s+k}^\vee)
$$
for all $k = 1,\ldots,N$, as desired.
\end{proof}

Let $\Pi = (A_0,A_1,\ldots,A_s) \in \AP(\lm)$ and $\widetilde{\Pi} = (A_0,A_1,\ldots,A_s,A_{s+1},\ldots,A_{s+N}) \in \widetilde{\AP}(\lm)$ be as before. Let $J \in \clA(\Pi)$. By \cite[Theorem 6.4]{BFP99}, there exists a unique saturated chain
$$
\ol{w}(J) \rightarrow \ol{w}(J) s_{\gamma_{i_1}} \rightarrow \cdots \rightarrow \ol{w}(J) s_{\gamma_{i_1}} \cdots s_{\gamma_{i_{N-t}}} = w_\circ
$$
such that $1 \leq i_1 < \cdots < i_{N-t} \leq N$. Thus, we obtain a bijection
$$
\Phi : \clA(\Pi) \rightarrow \clA(\widetilde{\Pi});\ \{ j_1, \ldots, j_t \} \mapsto \{ j_1, \ldots j_t, s + i_1, s + i_2 \ldots, s + i_{N-t} \}.
$$
The inverse map is given by $\tilde{J} \mapsto \tilde{J} \cap \{ 1,\ldots,s \}$.

\begin{prop}
Let $\Pi, \widetilde{\Pi}$ be as above. Then, the bijection $\Phi : \clA(\Pi) \rightarrow \clA(\widetilde{\Pi})$ commutes with $\wt,\Etil_p,\Ftil_p$, $p \in I$; here we understand $\Phi(0) = 0$. Consequently, $\clA(\widetilde{\Pi})$ equipped with the maps $\wt,\Etil_p,\Ftil_p$, $p \in I$ is a crystal isomorphic to $\clB(\lm)$, and $\Phi : \clA(\Pi) \rightarrow \clA(\widetilde{\Pi})$ is an isomorphism of crystals.
\end{prop}

\begin{proof}
Let $J = \{ j_1, \ldots, j_t \} \in \clA(\Pi)$, $p \in I$. We use the notation above. In particular, we write $\Phi(J) = \{ j_1, \ldots, j_t, s+i_1, s+i_2, \ldots, s+i_{N-t} \}$. In terms of galleries, let us write
$$
J = (A^J_0,F^J_1,A^J_1,\ldots,F^J_s,A^J_s,\lm^J).
$$
Then, $\Phi(J)$ is of the form
$$
(A^J_0,F^J_1,A^J_1,\ldots,F^J_s,A^J_s,F^J_{s+1},A^J_{s+1},\ldots,F^J_{s+N},A^J_{s+N},\mu^J).
$$
For each $i \in \{ 1,\ldots,s+N \}$, let $\beta^J_i \in \Phi^+$ and $l^J_i \in \Z$ be such that $F_i^J \subset H_{\beta^J_i,l^J_i}$.

First, we compute $\wt(\Phi(J))$. By definition, we have
$$
\wt(\Phi(J)) = -w(\Phi(J))(-\lm) = -w(J) s_{F_{s+i_1}} \cdots s_{F_{s+i_{N-t}}}(-\lm).
$$
Recall from Lemma \ref{added facet contains lm} that the facets $F_{s+i_1},\ldots,F_{s+i_{N-t}}$ contain $-\lm$. Therefore, the corresponding affine reflections stabilize $-\lm$. Hence, we obtain
$$
\wt(\Phi(J)) = -w(J)(-\lm) = \wt(J),
$$
as desired.

Next, we compute $\Ftil_p(\Phi(J))$. Obviously, $I(\Phi(J),p) = I(J,p) \sqcup I'$ for some subset $I'$ of $\{ s+1,\ldots,s+N \}$. Since the facets $F_{s+k}$, $k = 1,\ldots,N$ contain $-\lm$, the facets $F^J_{s+k}$ contain $\mu^J$. Hence, for each $i' \in I'$, we have
$$
l^J_{i'} = (\mu^J,\alpha_p^\vee) = (\lm^J,\alpha_p^\vee).
$$
Here, we used $\mu^J = -\wt(\Phi(J)) = -\wt(J) = \lm^J$. Therefore, we obtain
$$
L(\Phi(J),p) = L(J,p), \AND M(\Phi(J),p) = M(J,p).
$$

Now, we have three possibilities:
\begin{enumerate}
\item $M(\Phi(J),p) = M(J,p) = 0$. In this case, we have
$$
\Ftil_p(\Phi(J)) = 0, \qu \Ftil_p(J) = 0,
$$
and hence $\Ftil_p(\Phi(J)) = \Phi(\Ftil_p(J))$.
\item $M(\Phi(J),p) = M(J,p) < 0$ and $\{ i \in I(J,p) \mid l^J_i = M(J,p) \} \neq \emptyset$. In this case, we have $\{ i \in I(\Phi(J),p) \mid l^J_i = M(\Phi(J),p) \} \neq \emptyset$, and therefore, $m_F := m_F(\Phi(J),p) = m_F(J,p)$, $k_F := k_F(\Phi(J),p) = k_F(J,p)$. Hence,
$$
\Ftil_p(\Phi(J)) = (\Phi(J) \setminus \{m_F\}) \sqcup \{k_F\} = \Ftil_p(J) \sqcup \{ s+i_1,\ldots,s+i_{N-t} \}.
$$
Since $\ol{w}(\Ftil_p(J)) = \ol{w}(J)$ (by Proposition \ref{Properties of alcove paths model} \eqref{Properties of alcove paths model 6}), it follows that $\Ftil_p(J) \sqcup \{ s+i_1,\ldots,s+i_{N-t} \} = \Phi(\Ftil_p(J))$. This shows $\Ftil_p(\Phi(J)) = \Phi(\Ftil_p(J))$.
\item $M(\Phi(J),p) = M(J,p) < 0$ and $\{ i \in I(J,p) \mid l^J_i = M(J,p) \} = \emptyset$. In this case, we have $M(J,p) = (\lm^J,\alpha_p^\vee)$, and $\Ftil_p(J) = J \sqcup \{ k' \}$, where $k' := \max I(J,p) \in I(J,p) \cap \{ j_t+1,j_t+2,\ldots,s \}$.

Let us show that $I'$ is not empty. Recall that we have a saturated chain
$$
\ol{w}(J) \rightarrow \ol{w}(J)s_{\gamma_{i_1}} \rightarrow \cdots \rightarrow \ol{w}(J) s_{\gamma_{i_1}} \cdots s_{\gamma_{i_{N-t}}} = w_\circ.
$$
By Proposition \ref{Properties of alcove paths model} \eqref{Properties of alcove paths model 7}, we have $s_p \ol{w}(J) > \ol{w}(J)$. Since $s_p w_\circ < w_\circ$, we can take the minimal $k \in \{ 1,\ldots,N-t \}$ such that $s_p \ol{w}(J) s_{\gamma_{i_1}} \cdots s_{\gamma_{i_k}} < \ol{w}(J) s_{\gamma_{i_1}} \cdots s_{\gamma_{i_k}}$. Then, applying (the contraposition of) \cite[Corollary 2.2.8 (i)]{BB05} to $s = s_p$, $t = s_{\ol{w}(J)s_{\gamma_{i_1}} \cdots s_{\gamma_{i_{k-1}}}(\gamma_{i_k})}$, and $w = \ol{w}(J)s_{\gamma_{i_1}} \cdots s_{\gamma_{i_{k-1}}}$, we obtain $\ol{w}(J) s_{\gamma_{i_1}} \cdots s_{\gamma_{i_{k-1}}}(\gamma_{i_k}) = \alpha_p$. This implies that $i_k \in I'$.

Set $m' := \min I'$. Since $l^J_{i'} = (\lm^J,\alpha_p^\vee) = M(\Phi(J),p)$ for all $i' \in I'$, we see that
\begin{align}
\begin{split}
k_F &:= \max(I(\Phi(J),p) \cap \{ 1,\ldots,m'-1 \}) \\
&= \max I(J,p) = k',
\end{split} \nonumber
\end{align}
and hence,
\begin{align}
\begin{split}
\Ftil_p(\Phi(J)) &= (\Phi(J) \setminus \{ m' \}) \sqcup \{ k' \} \\
&= \Ftil_p(J) \sqcup (\{ s+i_1,\ldots,s+i_{N-t} \} \setminus \{ m' \}).
\end{split} \nonumber
\end{align}

Let us write $\Ftil_p(J) = \{ j'_1 < \cdots < j'_t < j'_{t+1} \}$ and $m' = s+i_m$ for some $1 \leq m \leq N-t$. Then,
$$
e \rightarrow s_{\beta_{j'_1}} \rightarrow \cdots \rightarrow s_{\beta_{j'_1}} \cdots s_{\beta_{j'_{t+1}}} = \ol{w}(\Ftil_p(J)) = s_p \ol{w}(J)
$$
is a saturated chain from $e$ to $s_p\ol{w}(J)$. Now, we show that
$$
s_p\ol{w}(J) \rightarrow s_p\ol{w}_J s_{\gamma_{i_1}} \rightarrow \cdots \rightarrow s_p \ol{w}(J)s_{\gamma_{i_1}} \cdots s_{\gamma_{i_{m-1}}} = \ol{w}(J)s_{\gamma_{i_1}} \cdots s_{\gamma_{i_m}}
$$
is a saturated chain from $s_p \ol{w}(J)$ to $\ol{w}(J)s_{\gamma_{i_1}} \cdots s_{\gamma_{i_m}}$. Let $1 \leq k \leq m-1$ and set $w_k := \ol{w}(J)s_{\gamma_{i_1}} \cdots s_{\gamma_{i_{k-1}}}$. Assume that $s_p\ol{w}(J) \rightarrow \cdots \rightarrow s_p \ol{w}(J) s_{\gamma_{i_1}} \cdots s_{\gamma_{i_{k-1}}} = s_p w_k$ is a saturated chain. We know that
$$
w_k \rightarrow w_k s_{\gamma_{i_k}} \AND w_k \rightarrow s_p w_k
$$
are saturated. Then, by \cite[Corollary 2.2.8 (i)]{BB05} again, both
$$
w_k \rightarrow w_k s_{\gamma_{i_k}} \rightarrow s_p w_k s_{\gamma_{i_k}} \AND w_k \rightarrow s_p w_k \rightarrow s_p w_k s_{\gamma_{i_k}}
$$
are saturated. This implies that $s_p\ol{w}(J) \rightarrow \cdots \rightarrow s_p \ol{w}(J) s_{\gamma_{i_1}} \cdots s_{\gamma_{i_{k}}} = s_p w_{k+1}$ is saturated. Then, by induction on $k$, one can prove the claim.

This far, we have obtained saturated chains from $e$ to $s_p \ol{w}(J)$, from $s_p \ol{w}(J)$ to $\ol{w}(J)s_{\gamma_{i_1}} \cdots s_{\gamma_{i_m}}$, and from $\ol{w}(J)s_{\gamma_{i_1}} \cdots s_{\gamma_{i_m}}$ to $\ol{w}(J)s_{\gamma_{i_1}} \cdots s_{\gamma_{i_{N-t}}} = w_\circ$. Concatenating these chains, we obtain a saturated chain from $e$ to $w_\circ$, which implies
$$
\Ftil_p(\Phi(J)) = \Phi(\Ftil_p(J)).
$$
\end{enumerate}

The assertion concerning $\Etil_p$ is proved similarly.
\end{proof}

\begin{cor}\label{Nonempty}
Let $\Pi \in \AP(\lm)$, $J \in \clA(\widetilde{\Pi})$, $p \in I$. If $\Ftil_p(J) \neq 0$ (resp., $\Etil_p(J) \neq 0$), then $\{ i \in I(J,p) \mid l^J_i = M(J,p) \} \neq \emptyset$ (resp., $k_E \neq k'$). Also, we have $\vphi_p(J) = M(J,p)$ and $\vep_p(J) = (-\wt(J),\alpha_p^\vee)-M(J,p)$.
\end{cor}

Let $\Pi = (A_0,\ldots,A_u) \in \widetilde{\AP}(\lm)$. We aim to show that $\clA(\Pi)$ admits a crystal structure isomorphic to $\clB(\lm)$ by relating $\clA(\Pi)$ with $\clA(\Pi'')$ for a certain alcove path $\Pi'' \in \AP(\lm+\rho)$, where $\rho \in \Lm^+$ denotes half the sum of positive roots. To do so, we need the following.

\begin{lem}
There exists an alcove path $\Pi' = (A'_0,A'_1,\ldots,A'_r) \in \AP(\rho)$ such that $A'_N = w_\circ A_\circ$.
\end{lem}

\begin{proof}
Let $R := \{ (\alpha,l) \mid \alpha \in \Phi^+,\ -(\rho,\alpha^\vee) < l \leq 0 \} \in \Phi^+ \times \Z$. Fix a total order on $I$, and identify $I$ with $\{ 1,\ldots,|I| \}$. Consider the map $v : R \rightarrow \Q^{|I|+1}$ defined by
$$
v(\alpha,l) := \frac{1}{\sum_{i \in I} c_i}(-l,c_1,\ldots,c_{|I|}),
$$
where $c_i \in \Z$ is such that $\alpha^\vee = \sum_{i \in I} c_i \alpha_i^\vee$. By the argument in \cite[Section 4]{LP08}, this map is injective, and there exists a reduced alcove path $\Pi' = (A'_0,A'_1,\ldots,A'_r) \in \AP(\rho)$ such that if we define $(\beta_i,l_i) \in \Phi^+ \times \Z$ by the condition that the common facet of $A'_{i-1}$ and $A'_i$ lies in the hyperplane $H_{\beta_i,l_i}$, then $\{ (\beta_i,l_i) \mid 1 \leq i \leq r \} = R$, and $j < k$ if and only if $v(\beta_j,l_j) <_{\lex} v(\beta_k,l_k)$, where $\leq_{\lex}$ denotes the lexicographic order on $\Q^{|I|+1}$.

Note that $\{ (\beta_i,l_i) \mid 1 \leq i \leq N \} = \{ (\alpha,0) \mid \alpha \in \Phi^+ \}$. This implies that $w := s_{\beta_1} \cdots s_{\beta_N} \in W$ and $\ell(w) = N = \ell(w_\circ)$. Therefore, it follows that $A'_N = w_\circ A_\circ$, as desired.
\end{proof}

Now, we define an alcove path $\Pi''$ by
$$
\Pi'' := (A_0,\ldots,A_u, A'_{N+1} - \lm, A'_{N+2} - \lm, \ldots, A'_{r} - \lm).
$$
Note that we have $A_u = w_\circ A-\lm = A'_N -\lm$.

\begin{lem}
$\Pi''$ is a reduced alcove path from $A_\circ$ to $A_\circ - \lm - \rho$.
\end{lem}

\begin{proof}
Let us write $\Pi'' = (A''_0,\ldots,A''_{u+r-N})$, and let $(\beta_i,l_i) \in \Phi^+ \times \Z$ be such that the common facet of $A''_{i-1}$ and $A''_i$ lies in the hyperplane $H_{\beta_i,l_i}$. By Proposition \ref{Characterization of chain}, it suffices to show that the sequence $(\beta_1,\ldots,\beta_{u+r-N})$ satisfies the conditions there.

Recall $\Pi' = (A'_0,\ldots,A'_r) \in \AP(\rho)$. Let $(\beta'_i,l'_i) \in \Phi^+ \times \Z$ be such that the common facet of $A'_{i-1}$ and $A'_i$ lies in the hyperplane $H_{\beta'_i,l'_i}$. Note also that $\beta'_{N+j} = \beta_{u+j}$ for all $j = 1,\ldots,r-N$.

First, let $\beta \in \Phi^+$. Since $(A'_0,\ldots,A'_N) \in \widetilde{\AP}(0)$, by Proposition \ref{Characterization of extended chain}, we have
$$
\sharp\{ i \in \{ 1,\ldots,N \} \mid \beta'_i = \beta \} = 1.
$$
Then, applying Proposition \ref{Characterization of chain} to $\Pi'$, we obtain
\begin{align}
\begin{split}
&\sharp\{ i \in \{ u+1,\ldots,u+r-N \} \mid \beta_i = \beta \} \\
&\qu = \sharp\{ i \in \{ N+1,\ldots,r \} \mid \beta'_i = \beta \} \\
&\qu = \sharp\{ i \in \{ 1,\ldots,r \} \mid \beta'_i = \beta \} - \sharp\{ i \in \{ 1,\ldots,N \} \mid \beta'_i = \beta \} \\
&\qu = (\rho,\beta^\vee) - 1.
\end{split} \nonumber
\end{align}
On the other hand, since $(A''_0,\ldots,A''_{u}) = (A_0,\ldots,A_u) \in \widetilde{\AP}(\lm)$, by Proposition \ref{Characterization of extended chain}, we see that
$$
\sharp\{ i \in \{ 1,\ldots,u \} \mid \beta_i = \beta \} = (\lm,\beta^\vee) + 1.
$$
Combining the results above, we obtain
$$
\sharp\{ i \in \{ 1,\ldots,u+r-N \} \mid \beta_i = \beta \} = (\lm+\rho,\beta^\vee).
$$
This implies the first condition.

The second condition can be verified in a similar way to the proof of Lemma \ref{extendable of alcove path}. Hence, we omit it.
\end{proof}

For an admissible subset $J = \{ j_1, \ldots, j_N \} \in \clA(\Pi)$, the set $\Psi(J) := J$, regarded as a subset of $\{ 1,\ldots,u+r-N \}$, is a member of $\clA(\Pi'')$. Clearly, this gives an injection
$$
\Psi : \clA(\Pi) \rightarrow \clA(\Pi'').
$$

By definition, we have
$$
\wt(\Psi(J)) = -w(\Psi(J))(-\lm-\rho) = -w(J)(-\lm-\rho).
$$
Recall from Remark \ref{Remark on w(J)} that there exists $\nu \in \Lm$ such that $w(J) = t_\nu w_\circ$. Hence,
\begin{align}
\begin{split}
-w(J)(-\lm-\rho) &= -w_\circ(-\lm-\rho) - \nu \\
&= -w_\circ(-\lm) - \nu - w_\circ(-\rho) \\
&= -w(J)(-\lm) - \rho = \wt(J) - \rho.
\end{split} \nonumber
\end{align}
This shows that
\begin{align}\label{Weight of Psi(J)}
\wt(\Psi(J)) = \wt(J) - \rho.
\end{align}

\begin{lem}\label{M(Psi(J),p)}
Let $J \in \clA(\Pi)$, $p \in I$. Then, we have either $M(\Psi(J),p) = M(J,p) + 1 = (-\wt(J),\alpha_p^\vee) + 1$ or $M(\Psi(J),p) = M(J,p) = l^J_i$ for some $i \in \Itil(J,p)$. Moreover, if $M(\Psi(J),p) = M(J,p) + 1$, then $\{ i \in I(J,p) \mid l^J_i = M(J,p) \} = \emptyset$.
\end{lem}

\begin{proof}
By Proposition \ref{Properties of alcove paths model} \eqref{Properties of alcove paths model 3}, we have
$$
M(\Psi(J),p) = \min(\{ l^{\Psi(J)}_i \mid i \in \Itil(\Psi(J),p) \} \cup \{ (-\wt(\Psi(J)),\alpha_p^\vee) \}).
$$
It is clear that $\Itil(\Psi(J),p) = \Itil(J,p)$, and that $l^{\Psi(J)}_i = l^J_i$ for all $i \in I(J,p)$. Also, by equation \eqref{Weight of Psi(J)}, we have
$$
(-\wt(\Psi(J)),\alpha_p^\vee) = (-\wt(J)+\rho,\alpha_p^\vee) = (-\wt(J),\alpha_p^\vee) + 1.
$$
On the other hand,
$$
M(J,p) = \min(\{ l^J_i \mid i \in I(J,p) \} \cup \{ (-\wt(J),\alpha_p^\vee) \}).
$$

Recall that we have $M(\Psi(J),p) \leq (-\wt(\Psi(J)),\alpha_p^\vee)$. Suppose first that
$$
M(\Psi(J),p) = (-\wt(\Psi(J)),\alpha_p^\vee) = (-\wt(J), \alpha^\vee_p) + 1.
$$
Then, we have $l^J_i > (-\wt(J),\alpha_p^\vee)$ for all $i \in I(J,p)$. This implies that
$$
M(J,p) = (-\wt(J),\alpha_p^\vee) = M(\Psi(J),p) - 1,
$$
and that $\{ i \in I(J,p) \mid l^J_i = M(J,p) \} = \emptyset$.

Next, suppose that $M(\Psi(J),p) < (-\wt(\Psi(J)),\alpha_p^\vee)$. Then, by above, there exists $i \in \Itil(J,p)$ such that $l^J_i = M(\Psi(J),p)$. This implies that $M(J,p) = l^J_i$. Thus, the proof completes.
\end{proof}

Now, recall from Subsection \ref{subsection, Yang-Baxter move} the notion of Yang-Baxter moves. Since it is equivalent to braid moves on the reduced expression of an element in $W_{\aff}$, we can consider the Yang-Baxter moves on $\widetilde{\AP}(\lm)$, as well as on $\AP(\lm)$. Such Yang-Baxter moves give rise to bijections among $\clA(\Pi)$'s, $\Pi \in \widetilde{\AP}(\lm)$.

\begin{theo}
Let $\lm \in \Lm^+$, $\Pi \in \AP(\lm)$, $\Pi_2 \in \widetilde{\AP}(\lm)$. Set $\Pi_1 \in \widetilde{\AP}(\lm)$ to be $\widetilde{\Pi}$ constructed in Lemma \ref{extendable of alcove path}. Then each sequence of Yang-Baxter moves transforming $\Pi_1$ into $\Pi_2$ induces a bijection $Y : \clA(\Pi_1) \rightarrow \clA(\Pi_2)$ commuting with $\wt,\Etil_p,\Ftil_p$ for all $p \in I$; here we understand $Y(0) = 0$.
\end{theo}

\begin{proof}
Let $\Pi_1''$ and $\Pi_2''$ be reduced alcove paths from $A_\circ$ to $A_{\circ}-\lm-\rho$ constructed from $\Pi_1$ and $\Pi_2$ in the same way as above, respectively. Then, a sequence of Yang-Baxter moves transforming $\Pi_1$ into $\Pi_2$ also makes $\Pi_1''$ into $\Pi''_2$. Hence, the map $Y$ induces an isomorphism $\clA(\Pi_1'') \rightarrow \clA(\Pi_2'')$ of crystals (by Subsection \ref{subsection, Yang-Baxter move}), and the following diagram commutes:
$$
\xymatrix{
\clA(\Pi_1'') \ar[r]^Y & \clA(\Pi_2'') \\
\clA(\Pi_1) \ar[r]_Y \ar[u]^\Psi & \clA(\Pi_2) \ar[u]_\Psi
}
$$

Let $J \in \clA(\Pi_1)$. Then, by equation \eqref{Weight of Psi(J)}, we have
$$
\wt(J) = \wt(\Psi(J)) + \rho = \wt(Y(\Psi(J))) + \rho = \wt(\Psi(Y(J))) + \rho = \wt(Y(J)),
$$
which shows the commutativity of $Y$ and $\wt$.

Now, we prove that $Y$ commutes with $\Ftil_p$ for all $p \in I$. To begin with, suppose that $\Ftil_p(J) = 0$. This implies that $M(J,p) = 0$. By Lemma \ref{M(Psi(J),p)} and the fact that $M(\Psi(J),p) \leq 0$ (Proposition \ref{Properties of alcove paths model} \eqref{Properties of alcove paths model 2}), we must have $M(\Psi(J),p) = 0$. Since $Y : \clA(\Pi''_1) \rightarrow \clA(\Pi''_2)$ is an isomorphism of crystals, we have
$$
0 = M(\Psi(J),p) = -\vphi_p(\Psi(J)) = -\vphi_p(Y(\Psi(J))) = M(Y(\Psi(J)),p) = M(\Psi(Y(J)),p),
$$
here we used Proposition \ref{Properties of alcove paths model} \eqref{Properties of alcove paths model 4}. Again, by Lemma \ref{M(Psi(J),p)}, we have either $M(Y(J),p) = 0$ or $M(Y(J),p) = -1$. If $M(Y(J),p) = -1$, then we have
$$
(-\wt(Y(J)),\alpha_p^\vee) = -1.
$$
This contradicts that
$$
0 = M(J,p) \leq (-\wt(J),\alpha_p^\vee) = (-\wt(Y(J)), \alpha_p^\vee).
$$
Therefore, we see that
$$
M(Y(J),p) = 0,
$$
and hence,
$$
\Ftil_p(Y(J)) = 0 = \Ftil_p(J),
$$
as desired.

Next, suppose that $M(J,p) < 0$. Recall from Corollary \ref{Nonempty} that
$$
\{ i \in I(J,p) \mid l^J_i = M(J,p) \} \neq \emptyset.
$$
Then, by Lemma \ref{M(Psi(J),p)}, there exists $i \in \Itil(J,p)$ such that $M(\Psi(J),p) = M(J,p) = l^J_i$. Now, it is clear that
$$
\Ftil_p(\Psi(J)) = \Psi(\Ftil_p(J)).
$$
In particular, we have
$$
\ol{w}(\Ftil_p(\Psi(J))) = \ol{w}(\Psi(\Ftil_p(J))) = w_\circ = \ol{w}(\Psi(J)).
$$
Since $Y : \clA(\Pi''_1) \rightarrow \clA(\Pi''_2)$ is an isomorphism of crystals, we compute as
\begin{align}
\begin{split}
\ol{w}(\Ftil_p(\Psi(Y(J)))) &= \ol{w}(Y(\Ftil_p(\Psi(J)))) = \ol{w}(\Ftil_p(\Psi(J))) = \ol{w}(\Psi(J)) = \ol{w}(Y(\Psi(J))) = \ol{w}(\Psi(Y(J))).
\end{split} \nonumber
\end{align}
By Proposition \ref{Properties of alcove paths model} \eqref{Properties of alcove paths model 6}--\eqref{Properties of alcove paths model 7}, this implies that there exists $k \in \Itil(\Psi(Y(J)),p)$ such that $M(\Psi(Y(J)),p) = l^{\Psi(Y(J))}_k$. Hence, by Lemma \ref{M(Psi(J),p)} again, we see that
$$
M(Y(J),p) = M(\Psi(Y(J)),p) = l^{Y(J)}_k.
$$
Therefore, we have
$$
\Ftil_p(\Psi(Y(J))) = \Psi(\Ftil_p(Y(J))),
$$
and hence,
$$
\Psi(\Ftil_p(Y(J))) = \Ftil_p(\Psi(Y(J))) = \Ftil_p(Y(\Psi(J))) = Y(\Ftil_p(\Psi(J))) = Y(\Psi(\Ftil_p(J))) = \Psi(Y(\Ftil_p(J))).
$$
Since $\Psi$ is injective, this implies that
$$
\Ftil_p(Y(J)) = Y(\Ftil_p(J)),
$$
as desired.

The commutativity of $Y$ and $\Etil_p$ can be proved similarly. Thus, the proof completes.
\end{proof}

\begin{cor}\label{extended alcove path is B(lm)}
For each $\Pi \in \widetilde{\AP}(\lm)$, the set $\clA(\Pi)$ equipped with maps $\wt,\Etil_p,\Ftil_p$, $p \in I$ is a crystal isomorphic to $\clB(\lm)$. Moreover, the Yang-Baxter moves induce isomorphisms of crystals among $\clA(\Pi)$'s, $\Pi \in \widetilde{\AP}(\lm)$.
\end{cor}

\begin{cor}\label{Properties of extended alcove paths model}
Let $\Pi = (A_0,\ldots,A_u) \in \widetilde{\AP}(\lm)$, $J = \{ j_1,\ldots, j_N \} \in \clA(\Pi)$, $p \in I$. Then, the following hold:
\begin{enumerate}
\item\label{Properties of extended alcove paths model 1} $M(J,p) \leq 0$.
\item $\vphi_p(J) = M(J,p)$.
\item\label{Properties of extended alcove paths model 5} $\vep_p(J) = (-\wt(J),\alpha_p^\vee) - M(J,p)$.
\item\label{Properties of extended alcove paths model 2} We have
$$
\Ftil_p(J) = \begin{cases}
0 \qu & \IF M(J,p) = 0, \\
(J \setminus \{ j_m \}) \sqcup \{k\} \qu & \IF M(J,p) < 0,
\end{cases}
$$
where $m \in \{ 1,\ldots,N \}$ is such that $j_m = \min \{ i \in \Itil(J,p) \mid l^J_i = M(J,p) \}$, and $k = \max (I(J,p) \cap \{ 1,2,\ldots,j_m-1 \})$.
\item\label{Properties of extended alcove paths model 3} We have
$$
\Etil_p(J) = \begin{cases}
0 \qu & \IF M(J,p) = (-\wt(J),\alpha_p^\vee), \\
(J \setminus \{ j_k \}) \sqcup \{m\} \qu & \IF M(J,p) < (-\wt(J),\alpha_p^\vee),
\end{cases}
$$
where $k \in \{ 1,\ldots,N \}$ is such that $j_k = \max \{ i \in \Itil(J,p) \mid l^J_i = M(J,p) \}$, and $m = \min (I(J,p) \cap \{ j_k+1,j_k+2,\ldots,u \})$.
\end{enumerate}
\end{cor}

\subsection{Almost $\bfi_A$-decreasing subsets}\label{Almost decreasing subsets}
Let $\Pi = (A_0,\ldots,A_u) \in \widetilde{\AP}(\lm)$ with $\Gamma(\Pi) = (\beta_1,\ldots,\beta_u)$. For $1 \leq i \leq u$, set
$$
N(i) = N(i;\Pi) := \sharp \{ k > i \mid \beta_k = \beta_i \}.
$$

Recall from Proposition \ref{Characterization of extended chain} that the number of occurrences of $\beta_i$ in $\Gamma(\Pi)$ is equal to $(\lm,\beta_i^\vee) + 1$ and that $l_i = -\sharp\{ j < i \mid \beta_j = \beta_i \}$. Then, we have
\begin{align}\label{l, N, and lm}
-l_i + N(i) = (\lm,\beta_i^\vee).
\end{align}

From now on, assume that our root system is of type $A_{n-1}$. In this case, we have $\alpha^\vee = \alpha$ for all $\alpha \in \Phi$. Hence, we do not distinguish them. Recall the reduced word $\bfi_A = (i_1,\ldots,i_N)$ from Subsection \ref{GT patterns}.

\begin{defi}\normalfont
Let $\Pi \in \widetilde{\AP}(\lm)$ and $J = \{ j_1, \ldots, j_N \} \in \clA(\Pi)$. We say that $J$ is an almost $\bfi_A$-decreasing subset if it satisfies the following:
\begin{itemize}
\item $\{ \beta_j \mid j \in J \} = \Phi^+$.
\item For each $1 \leq k < l \leq N$, we have either $(\beta_{j_k},\beta_{j_l}) = 0$ or $\beta_{j_l} <_{\bfi_A} \beta_{j_k}$.
\end{itemize}
For an almost $\bfi_A$-decreasing subset $J = \{ j_1,\ldots, j_N \}$ and $1 \leq i < j \leq n$, we set $N_{i,j} = N_{i,j}(J) := N(j_k)$ if $\beta_{j_k} = \eps_i-\eps_j$. Also, we set $N(J) := (N_{i,j})_{1 \leq i < j \leq n}$.
\end{defi}

Now, we are ready to state our main result in this paper.

\begin{theo}\label{Main Theorem}
Let $\lm \in \Lm^+$ and $\Pi \in \widetilde{\AP}(\lm)$.
\begin{enumerate}
\item\label{Main Theorem 1} Each $J \in \clA(\Pi)$ can be transformed by a sequence of Yang-Baxter moves into an admissible subset $J'$ such that $\Etil_{i_k}^l \Etil_{i_{k-1}}^{\max} \cdots \Etil_{i_1}^{\max}(J')$ is an almost $\bfi_A$-decreasing subset for all $1 \leq k \leq N$ and $0 \leq l \leq \vep_{i_k}(\Etil_{i_{k-1}}^{\max} \cdots \Etil_{i_1}^{\max}(J'))$.
\item\label{Main Theorem 2} Let $J \in \clA(\Pi)$, and $J'$ be as in \eqref{Main Theorem 1}. Then, $N(J')$ is independent of the choice of $J'$; due to this result, we may define $N_{i,j}(J) := N_{i,j}(J')$ for each $1 \leq i < j \leq n$, and $N(J) := N(J')$.
\item\label{Main Theorem 3} For each $J \in \clA(\Pi)$, the tuple $\bfa(J) = (a_{i,j}(J))_{1 \leq i \leq j \leq n}$ defined by
$$
a_{i,j}(J) = \lm_i - N_{i,j}(J) \qu (\text{we set $N_{i,j}(J) = 0$ if $i = j$})
$$
is a Gelfand-Tsetlin pattern of shape $\lm$. Moreover, this assignment gives rise to an isomorphism of crystals between $\clA(\Pi)$ and $\GT(\lm)$.
\end{enumerate}
\end{theo}

Before moving to a detailed discussion, we outline the proof. First, we construct a certain alcove path $\Pi(\lm) \in \widetilde{\AP}(\lm)$, and show that each $J \in \clA(\Pi(\lm))$ is an almost $\bfi_A$-decreasing subset. This, together with Corollary \ref{extended alcove path is B(lm)} proves item \eqref{Main Theorem 1}. Next, we prove item \eqref{Main Theorem 3} for $\Pi = \Pi(\lm)$ by comparing the $\bfi_A$-string data of $J$ and $\bfa(J)$. Finally, we prove item \eqref{Main Theorem 2} and \eqref{Main Theorem 3} by computing the $\bfi_A$-string datum of $J'$.

The rest of this subsection is devoted to investigating basic properties of almost $\bfi_A$-decreasing subsets.

Let $\Pi = (A_0,\ldots,A_u) \in \widetilde{\AP}(\lm)$ with $\Gamma(\Pi) = (\beta_1,\ldots,\beta_u)$. Let $J = \{ j_1,\ldots, j_N \} \in \clA(\Pi)$ be an almost $\bfi_A$-decreasing subset. Also, let us write $J = (A^J_0,F^J_1,A^J_1,\ldots,F^J_u,A^J_u,\lm^J)$ with $F^J_i \subset H_{\beta^J_i,l^J_i}$, $\beta^J_i \in \Phi^+$, $l^J_i \in \Z$. Recall that $\{ \beta_j \mid j \in J \} = \Phi^+$. Then, for each $1 \leq a < b \leq n$, there exists a unique $j_{a,b} \in J$ such that $\beta_{j_{a,b}} = \eps_a-\eps_b$. For notational simplicity, we write $\beta_{a,b} := \beta_{j_{a,b}}$, $\beta^J_{a,b} := \beta^J_{j_{a,b}}$, $l_{a,b} := l_{j_{a,b}}$, and $l^J_{a,b} := l^J_{j_{a,b}}$. We will use these notation whenever we consider an almost $\bfi_A$-decreasing subset.

\begin{prop}\label{Properties of almost decreasing subset}
Let $1 \leq a < b \leq n$, $p \in \{ 1,\ldots,n-1 \}$, and $q \in \{ 1,\ldots,p-1 \}$. Then, we have the following:
\begin{enumerate}
\item\label{Properties of almost decreasing subset 1} $j_{a+1,b},j_{a,b+1},j_{a+1,b+1} < j_{a,b}$.
\item\label{Properties of almost decreasing subset 2} $\beta^J_{a,b} = \alpha_{n-(b-a)}$.
\item\label{Properties of almost decreasing subset 3} $\Itil(J,p) = \{ j_{p,n}, j_{p-1,n-1}, \cdots, j_{1,n-p+1} \}$.
\item\label{Properties of almost decreasing subset 4} $l^J_{p,n} = N_{p,n} - N_{p+1,n} - (\lm_p - \lm_{p+1})$.
\item\label{Properties of almost decreasing subset 5} $l^J_{q,n-p+q} - l^J_{q+1,n-p+q+1} = N_{q,n-p+q} - N_{q+1,n-p+q} - N_{q,n-p+q+1} + N_{q+1,n-p+q+1}$.
\item\label{Properties of almost decreasing subset 6} $\wt(J) = w_\circ(\lm) - \sum_{m=1}^{n-1} \sum_{\substack{1 \leq c < d \leq n \\ d-c=n-m}} l_{c,d} \alpha_m$.
\item\label{Properties of almost decreasing subset 7} $(-\wt(J),\alpha_p) - l^J_{1,n-p+1} = N_{1,n-p+1} - N_{1,n-p}$.
\end{enumerate}
\end{prop}

\begin{proof}
Let us prove \eqref{Properties of almost decreasing subset 1}. Since $(a,b) <_{\bfi_A} (a+1,b), (a,b+1)$ and $(\beta_{a,b},\beta_{a+1,b}) = (\beta_{a,b},\beta_{a,b+1}) = 1$, we must have $j_{a+1,b},j_{a,b+1} < j_{a,b}$. Replacing $(a,b)$ by $(a+1,b)$, we obtain $j_{a+1,b+1} < j_{a+1,b}$. This implies $j_{a+1,b+1} < j_{a,b}$.

Next, let us prove \eqref{Properties of almost decreasing subset 2}. By the definitions of $\beta^J_{a,b}$ and $l^J_{a,b}$, we have
$$
H_{\beta^J_{a,b},l^J_{a,b}} = \left(\prod_{j \in J,\ j < j_{a,b}} s_{\beta_j,l_j}\right)(H_{\beta_{a,b},l_{a,b}}).
$$
Suppose that there exists $k \in J$ such that $k < j_{a,b}$ and $\beta_k <_{\bfi_A} \beta_{a,b}$. We can take the maximum $k$ with this property. Since $\beta_{a,b} \leq_{\bfi_A} \beta_{k'}$ for all $k' \in J$ such that $k < k' \leq j_{a,b}$, we have $\beta_k <_{\bfi_A} \beta_{k'}$. Hence, $(\beta_k,\beta_{k'}) = 0$ for all $k < k' \leq j_{a,b}$. Therefore,
$$
\left( \prod_{j \in J,\ j < j_{a,b}} s_{\beta_j,l_j} \right)(H_{\beta_{a,b},l_{a,b}}) = \left( \prod_{j \in J \setminus \{ k \},\ j < j_{a,b}} s_{\beta_j,l_j} \right)(H_{\beta_{a,b},l_{a,b}}).
$$
Repeating this procedure, we obtain
\begin{align}
\begin{split}
H_{\beta^J_{a,b},l^J_{a,b}} &= \left( \prod_{\substack{j \in J,\ j < j_{a,b} \\ \beta_{a,b} <_{\bfi_A} \beta_j}} s_{\beta_j,l_j} \right)(H_{\beta_{a,b},l_{a,b}}) \\
&= \left( \prod_{\substack{1 \leq c < d \leq n \\ j_{c,d} < j_{a,b} \AND (a,b) <_{\bfi_A} (c,d)}} s_{\beta_{c,d},l_{c,d}} \right)(H_{\beta_{a,b},l_{a,b}}),
\end{split} \nonumber
\end{align}
where the last product is taken in the decreasing order of $\leq_{\bfi_A}$.

First, consider the case when $b = n$. In this case, we have $(a,n) <_{\bfi_A} (c,d)$ if and only if $d = n$ and $c > a$. Since $(\beta_{c,n},\beta_{a,n}) = 1$ for all $c > a$, we see that $j_{c,n} < j_{a,n}$. By means of Lemma \ref{reflection of hyperplanes}, we compute as
\begin{align}
\begin{split}
H_{\beta^J_{a,n},l^J_{a,n}} &= \left( \prod_{a < c < n} s_{\beta_{c,n},l_{c,n}} \right)(H_{\beta_{a,n},l_{a,n}}) \\
&= \left( \prod_{a+1 < c < n} s_{\beta_{c,n},l_{c,n}} \right) s_{\beta_{a+1,n},l_{a+1,n}}(H_{\beta_{a,n},l_{a,n}}) \\
&= \left( \prod_{a+1 < c < n} s_{\beta_{c,n},l_{c,n}} \right) (H_{\beta_{a,a+1},l_{a,n}-l_{a+1,n}}) \\
&= H_{\beta_{a,a+1},l_{a,n}-l_{a+1,n}}.
\end{split} \nonumber
\end{align}
This implies that
$$
\beta^J_{a,n} = \beta_{a,a+1} = \eps_a-\eps_{a+1} = \alpha_a, \qu l^J_{a,n} = l_{a,n} - l_{a+1,n}.
$$

Next, consider the case when $b = a+1 < n$. In this case, we have $(a,a+1) <_{\bfi_A} (c,d) \leq_{\bfi_A} (a+1,a+2)$ if and only if $d = a+2$ and $c \leq a+1$. Noting that $(\beta_{i,a+2},\beta_{a,a+1}) = 0$ for all $i < a$, we compute as
\begin{align}
\begin{split}
&\left( \prod_{\substack{1 \leq c < d \leq n \\ j_{a+1,a+2} \leq j_{c,d} < j_{a,a+1} \AND (a,a+1) <_{\bfi_A} (c,d) \leq_{\bfi_A} (a+1,a+2)}} s_{\beta_{c,d},l_{c,d}} \right)(H_{\beta_{a,a+1},l_{a,a+1}}) \\
&= \left( \prod_{\substack{1 \leq c \leq a+1 \\ j_{a+1,a+2} \leq j_{c,a+2} < j_{a,a+1}}} s_{\beta_{c,a+2},l_{c,a+2}} \right)(H_{\beta_{a,a+1},l_{a,a+1}}) \\
&= s_{\beta_{a+1,a+2},l_{a+1,a+2}} s_{\beta_{a,a+2},l_{a,a+2}}(H_{\beta_{a,a+1},l_{a,a+1}}) \\
&= s_{\beta_{a+1,a+2},l_{a+1,a+2}}(H_{\beta_{a+1,a+2},-l_{a,a+1}+l_{a,a+2}}) \\
&= H_{\beta_{a+1,a+2},l_{a,a+1}-l_{a,a+2}+2l_{a+1,a+2}}.
\end{split} \nonumber
\end{align}
This implies that
$$
\beta^J_{a,a+1} = \beta^J_{a+1,a+2}, \AND l^J_{a,a+1} = l_{a,a+1} - l_{a,a+2} + l_{a+1,a+2} + l^J_{a+1,a+2}.
$$

Finally, let us consider the case when $b \neq n, a+1$. In a way similar to above, we see that
$$
\beta^J_{a,b} = \beta^J_{a+1,b+1}, \AND l^J_{a,b} = l_{a,b} - l_{a+1,b} - l_{a,b+1} + l_{a+1,b+1} + l^J_{a+1,b+1}.
$$

Summarizing, we obtain
$$
(\beta^J_{a,b},l^J_{a,b}) = \begin{cases}
(\alpha_a,l_{a,n} - l_{a+1,n}) \qu & \IF b = n, \\
(\beta^J_{a+1,b+1},l_{a,b} - l_{a+1,b} - l_{a,b+1} + l_{a+1,b+1} + l^J_{a+1,b+1}) \qu & \IF b < n,
\end{cases}
$$
where we understand $l_{c,d} = 0$ if $c = d$. In particular, we see that
$$
\beta^J_{a,b} = \beta^J_{n-b+a,n} = \alpha_{n-(b-a)}.
$$
This proves \eqref{Properties of almost decreasing subset 2}.

Now, \eqref{Properties of almost decreasing subset 3} is immediate from \eqref{Properties of almost decreasing subset 2}.

Let us prove \eqref{Properties of almost decreasing subset 4} and \eqref{Properties of almost decreasing subset 5}. Since
$$
-l_{i,j} + N_{i,j} = (\lm, \beta_{i,j}) = \lm_i-\lm_j,
$$
it follows that
$$
l^J_{a,b} = \begin{cases}
N_{a,n} - N_{a+1,n} - (\lm_a-\lm_{a+1}) \qu & \IF b = n, \\
N_{a,b} - N_{a+1,b} - N_{a,b+1} + N_{a+1,b+1} + l^J_{a+1,b+1} \qu & \IF b < n.
\end{cases}
$$
This implies \eqref{Properties of almost decreasing subset 4} and \eqref{Properties of almost decreasing subset 5}.

Let us prove \eqref{Properties of almost decreasing subset 6}. We compute as
$$
w(J) = s_{\beta_{j_1},l_{j_1}} \cdots s_{\beta_{j_N},l_{j_N}} = (t_{l_{j_1}\beta_{j_1}} s_{\beta_{j_1}}) \cdots (t_{l_{j_N}\beta_{j_N}} s_{\beta_{j_N}}) = t_\nu w_\circ,
$$
where
$$
\nu = \sum_{i \in J} l_i \beta^J_i.
$$
Here, we used the fact that
$$
s_\alpha t_{l\beta} = t_{l s_\alpha(\beta)} s_\alpha \qu \Forall \alpha,\beta \in \Phi, l \in \Z.
$$
By \eqref{Properties of almost decreasing subset 2}, we see that
$$
\nu = \sum_{m=1}^{n-1} \sum_{\substack{1 \leq c < d \leq n \\ d-c = n-m}} l_{c,d} \alpha_m.
$$
Hence, we obtain
$$
\wt(J) = -w(J)(-\lm) = -t_\nu w_\circ(-\lm) = w_\circ(\lm) - \sum_{m=1}^{n-1} \sum_{\substack{1 \leq c < d \leq n \\ d-c = n-m}} l_{c,d} \alpha_m,
$$
as desired.

Finally, let us prove \eqref{Properties of almost decreasing subset 7}. Using \eqref{Properties of almost decreasing subset 6}, we compute
$$
(\wt(J),\alpha_p) = (w_\circ(\lm),\alpha_p) - 2\sum_{\substack{1 \leq c < d \leq n \\ d-c=n-p}} l_{c,d} + \sum_{\substack{1 \leq c < d \leq n \\ d-c=n-p\pm 1}} l_{c,d}.
$$
Note that
$$
(w_\circ(\lm),\alpha_p) = -(\lm,\alpha_{n-p}) = -\lm_{n-p}+\lm_{n-p+1}.
$$
Also, we compute $l^J_{1,n-p+1}$ as follows:
\begin{align}
\begin{split}
l^J_{1,n-p+1} &= \sum_{i=1}^{p-1}(l^J_{i,n-p+i} - l^J_{i+1,n-p+i+1}) + l^J_{p,n} \\
&= \sum_{i=1}^{p-1}(l_{i,n-p+i} - l_{i+1,n-p+i} - l_{i,n-p+i+1} + l_{i+1,n-p+i+1}) + (l_{p,n} - l_{p+1,n}) \\
&= 2\sum_{\substack{1 \leq c < d \leq n \\ d-c=n-p}} l_{c,d} - \sum_{\substack{1 \leq c < d \leq n \\ d-c=n-p\pm 1}} l_{c,d} - l_{1,n-p+1} + l_{1,n-p}.
\end{split} \nonumber
\end{align}
Hence, we obtain
$$
(-\wt(J),\alpha_p) - l^J_{1,n-p+1} = \lm_{n-p} - \lm_{n-p+1} + l_{1,n-p+1} - l_{1,n-p} = N_{1,n-p+1} - N_{1,n-p},
$$
as desired. Thus, the proof completes.
\end{proof}

\begin{prop}\label{Properties of Ep(J)}
Let $p \in \{ 1,\ldots,n-1 \}$ and suppose that $\Etil_p(J) \neq 0$ and that $\Etil_p(J)$ is an almost $\bfi_A$-decreasing subset. Let us write $\Etil_p(J) = (J \setminus \{ j \}) \sqcup \{ m \}$ for some $j,m$.  Then, the following hold:
\begin{enumerate}
\item\label{Properties of Ep(J) 1} $j = j_{c,d}$, where $1 \leq c < d \leq n$ are such that $j_{c,d} = \max \{ j_{a,b} \mid b-a = n-p \AND l^J_{a,b} = M(J,p) \}$.
\item\label{Properties of Ep(J) 2} If $k \in J$ satisfies $j < k < m$, then $(\beta_j, \beta_k) = 0$.
\item\label{Properties of Ep(J) 5} $$
N_{i,j}(\Etil_p(J)) = \begin{cases}
N_{i,j} \qu & \IF (i,j) \neq (c,d), \\
N_{c,d}-1 \qu & \IF (i,j) = (c,d).
\end{cases}
$$
\end{enumerate}
\end{prop}

\begin{proof}
Item \eqref{Properties of Ep(J) 1} follows from Corollary \ref{Properties of extended alcove paths model} \eqref{Properties of extended alcove paths model 3}, and Proposition \ref{Properties of almost decreasing subset} \eqref{Properties of almost decreasing subset 3}.

Let us prove item \eqref{Properties of Ep(J) 2}. We have
$$
\{ \beta_k \mid k \in \Etil_p(J) \} = (\{ \beta_k \mid k \in J \} \setminus \{ \beta_j \}) \cup \{ \beta_m \}.
$$
Since both $J$ and $\Etil_p(J)$ are almost $\bfi_A$-decreasing subsets, we have
$$
\{ \beta_k \mid k \in J \} = \{ \beta_k \mid k \in \Etil_p(J) \} = \Phi^+.
$$
This implies that
$$
\beta_j = \beta_m.
$$
Let $k \in J$ be such that $j < k < m$. Assume contrary that $(\beta_j,\beta_k) \neq 0$. Since $J$ is an almost $\bfi_A$-decreasing subset, this implies that
$$
\beta_k <_{\bfi_A} \beta_j = \beta_m.
$$
On the other hand, since $\Etil_p(J) = (J \setminus \{j\}) \cup \{m\}$ is an almost $\bfi_A$-decreasing subset, we have
$$
\beta_m <_{\bfi_A} \beta_k.
$$
Thus, we obtain a contradiction. Hence, we conclude that
$$
(\beta_j,\beta_k) = 0.
$$



Let us prove item \eqref{Properties of Ep(J) 5}. It suffices to show that $m = \min \{ i > j \mid \beta_i = \beta_j \}$. Assume contrary that there exists $i > j$ such that $\beta_i = \beta_j$ and $i < m$. By \eqref{Properties of Ep(J) 2}, we see that
$$
\beta^J_i = \beta^J_j = \alpha_p.
$$
This implies that $i \in I(J,p)$. However, this contradicts that
$$
m = \min(I(J,p) \cap \{ j+1,\ldots,u \}).
$$
Thus, the proof completes.

\end{proof}

\subsection{Proof of Theorem \ref{Main Theorem} \eqref{Main Theorem 1}}
Let us construct the reduced alcove path $\Pi(\lm) \in \widetilde{\AP}(\lm)$. For $1 \leq i \leq n-1$, consider the following sequence of positive roots
\begin{align}
\begin{split}
\Gamma(i) := (&\eps_i-\eps_{n},\eps_i-\eps_{n-1},\ldots,\eps_i-\eps_{i+1}, \\
&\eps_{i-1}-\eps_{n},\eps_{i-1}-\eps_{n-1},\ldots,\eps_{i-1}-\eps_{i+1}, \\
&\ldots,\\
&\eps_1-\eps_{n},\eps_{1}-\eps_{n-1},\ldots,\eps_{1}-\eps_{i+1}),
\end{split} \nonumber
\end{align}
and set
\begin{align}
\begin{split}
\Gamma(\lm) := &\Gamma(n-1)^{\lm_{n-1}-\lm_n} \circ (\eps_{n-1}-\eps_n) \\
\circ &\Gamma(n-2)^{\lm_{n-2}-\lm_{n-1}} \circ (\eps_{n-2}-\eps_n,\eps_{n-2}-\eps_{n-1}) \\
\circ &\cdots \\
\circ &\Gamma(1)^{\lm_1-\lm_2} \circ (\eps_1-\eps_n,\eps_1-\eps_{n-1},\ldots,\eps_1-\eps_2),
\end{split} \nonumber
\end{align}
where $\circ$ means the concatenation of sequences of positive roots.

\begin{ex}\label{Example 431}\normalfont
Suppose that $n = 3$ and $\lm = (2,1,0)$. Then,
$$
\Gamma(\lm) = (\eps_2-\eps_3,\eps_1-\eps_3,\eps_2-\eps_3,\eps_1-\eps_3,\eps_1-\eps_2,\eps_1-\eps_3,\eps_1-\eps_2).
$$
\end{ex}

\begin{lem}\label{Gamma(lm) is reduced}
$\Gamma(\lm)$ is the sequence of positive roots of a reduced alcove path from $A_\circ$ to $w_\circ A_\circ - \lm$.
\end{lem}

\begin{proof}
Let us write $\Gamma(\lm) = (\beta_1,\ldots,\beta_u)$. It is straightforwardly verified that for each $i < j < k$, the subsequence of $\Gamma(\lm)$ consisting of $\eps_i-\eps_j,\eps_j-\eps_k,\eps_i-\eps_k$ is
$$
(\eps_j-\eps_k,\eps_i-\eps_k)^{\lm_j-\lm_k} \circ (\eps_j-\eps_k) \circ (\eps_i-\eps_k,\eps_i-\eps_j)^{\lm_i-\lm_j+1}.
$$
Then, the lemma follows from Proposition \ref{Characterization of extended chain}.
\end{proof}

Let $\Pi(\lm)$ denote the reduced alcove path from $A_\circ$ to $w_\circ A_\circ - \lm$ such that $\Gamma(\Pi(\lm)) = \Gamma(\lm)$.

\begin{ex}[{cf. Examples \ref{Example 232}, \ref{Example 415}, and \ref{Example 431}}]\normalfont
Suppose that $n = 3$ and $\lm = (2,1,0)$. Then, the elements of $\clA(\Pi(\lm))$ are
$$
\{ 1,2,5 \}, \{ 1,2,7 \}, \{ 1,4,5 \}, \{ 1,4,7 \}, \{ 1,6,7 \}, \{ 3,4,5 \}, \{ 3,4,7 \}, \{ 3,6,7 \}.
$$
These are all almost $\bfi_A$-decreasing subsets. The $N(J)=(N_{2,3}(J),N_{1,3}(J),N_{1,2}(J))$'s, $J \in \clA(\Pi(\lm))$ are
$$
(1,2,1), (1,2,0), (1,1,1), (1,1,0), (1,0,0), (0,1,1), (0,1,0), (0,0,0).
$$
\end{ex}

\begin{prop}\label{J(a)}
Let $\bfa = (a_{i,j})_{1 \leq i \leq j \leq n} \in \GT(\lm)$. Then, there exists a unique $J(\bfa) \in \clA(\Pi(\lm))$ such that $J(\bfa)$ is an almost $\bfi_A$-decreasing subset, and $N_{i,j}(J(\bfa)) = \lm_i-a_{i,j}$ for all $1 \leq i < j \leq n$.
\end{prop}

\begin{proof}
Let us write $\Gamma(\lm) = (\beta_1,\ldots,\beta_u)$. Let $1 \leq k < l \leq n$. By the definition of Gelfand-Tsetlin patterns, we have
$$
\lm_l = a_{l,l} \leq a_{l-1,l} \leq \cdots \leq a_{k+1,l} \leq a_{k,l} \leq a_{k,l-1} \leq \cdots \leq a_{k,k+1} \leq a_{k,k} = \lm_k.
$$
Hence, it follows that
$$
0 \leq \lm_k - a_{k,l} \leq \lm_k-\lm_l = (\lm,\eps_k-\eps_l).
$$
Since the number of occurrences of $\eps_k-\eps_l$ in $\Gamma(\lm)$ is equal to $(\lm,\eps_k-\eps_l)+1$, there exists a unique $i_{k,l} \in \{ 1,\ldots,u \}$ such that $\beta_{i_{k,l}} = \eps_k-\eps_l$ and $N(i_{k,l}) = \lm_k-a_{k,l}$.

Let $J(\bfa) := \{ j_1,\ldots,j_N \}$ with $j_1 < \cdots < j_N$ be the rearrangement of $\{ i_{k,l} \mid 1 \leq k < l \leq n \}$. Clearly, we have $\{ \beta_j \mid j \in J(\bfa) \} = \Phi^+$. We show that the total order on $\Phi^+$ defined by $\beta_{j_1} < \cdots < \beta_{j_N}$ is a reflection order. Let $1 \leq a,b,c \leq N$ be such that $\beta_{j_c} = \beta_{j_a} + \beta_{j_b}$. Exchanging $a$ and $b$ if necessary, we can write $\beta_{j_a} = \eps_i-\eps_j$, $\beta_{j_b} = \eps_j-\eps_k$, $\beta_{j_c} = \eps_i-\eps_k$ for some $1 \leq i < j < k \leq n$. As we have seen in the proof of Lemma \ref{Gamma(lm) is reduced}, the subsequence of $\Gamma(\lm)$ consisting of $\eps_i-\eps_j,\eps_j-\eps_k,\eps_i-\eps_k$ is
$$
(\eps_j-\eps_k,\eps_i-\eps_k)^{\lm_j-\lm_k} \circ (\eps_j-\eps_k) \circ (\eps_i-\eps_k,\eps_i-\eps_j)^{\lm_i-\lm_j+1}.
$$
This shows that the $N_{i,j}+1$-th $\eps_i-\eps_j$ from the right is on the right of the $N_{i,k}+1$-th $\eps_i-\eps_k$ from the right since we have
$$
N_{i,k} = \lm_i-a_{i,k} > \lm_i-a_{i,j} = N_{i,j}.
$$
This implies that $c < a$. Similarly, we see that $b < c$. Thus, we obtain
$$
b < c < a.
$$
Therefore, $\beta_{j_1} < \cdots < \beta_{j_N}$ is a reflection order. In particular, $J(\bfa) \in \clA(\Pi(\lm))$.

Finally, we prove that $J(\bfa)$ is an almost $\bfi_A$-decreasing subset. Let $1 \leq a < b \leq N$ be such that $(\beta_{j_a},\beta_{j_b}) \neq 0$. Then, there exist $1 \leq i < j < k \leq n$ such that one of the following holds:
\begin{itemize}
\item $\beta_{j_a} = \eps_i-\eps_j$ and $\beta_{j_b} = \eps_i-\eps_k$.
\item $\beta_{j_a} = \eps_i-\eps_j$ and $\beta_{j_b} = \eps_j-\eps_k$.
\item $\beta_{j_a} = \eps_i-\eps_k$ and $\beta_{j_b} = \eps_i-\eps_j$.
\item $\beta_{j_a} = \eps_i-\eps_k$ and $\beta_{j_b} = \eps_j-\eps_k$.
\item $\beta_{j_a} = \eps_j-\eps_k$ and $\beta_{j_b} = \eps_i-\eps_j$.
\item $\beta_{j_a} = \eps_j-\eps_k$ and $\beta_{j_b} = \eps_i-\eps_k$.
\end{itemize}
However, by argument above, only the third, fifth, or sixth can happen. In each case, we obtain $\beta_{j_b} <_{\bfi_A} \beta_{j_a}$. This proves that $J(\bfa)$ is an almost $\bfi_A$-decreasing subset.
\end{proof}

\begin{cor}
Each $J \in \clA(\Pi(\lm))$ is an almost $\bfi_A$-decreasing subset, and the map $\clA(\Pi(\lm)) \rightarrow \GT(\lm);\ J \mapsto \bfa(J)$ is bijective.
\end{cor}

\begin{proof}
By Proposition \ref{J(a)}, we obtain an injection $\GT(\lm) \rightarrow \clA(\Pi(\lm));\ \bfa \mapsto J(\bfa)$. This is indeed a bijection since we have $|\GT(\lm)| = |\clB(\lm)| = |\clA(\Pi(\lm))|$. It is easily verified that the inverse of this bijection is given by $J \mapsto \bfa(J)$. This proves the corollary.
\end{proof}

\begin{cor}\label{aij and nij}
Let $J \in \clA(\Pi(\lm))$. Then, for each $1 \leq i < j \leq n$, we have
$$
\lm_i-\lm_{i+1} + N_{i+1,j} \geq N_{i,j} \geq N_{i,j-1}.
$$
\end{cor}

\begin{proof}
Let $\bfa \in \GT(\lm)$ be such that $\lm_i-a_{i,j} = N_{i,j}$. Then, the assertion is equivalent to the defining condition for $\bfa$ that
$$
a_{i+1,j} \leq a_{i,j} \leq a_{i,j-1}.
$$
\end{proof}

\subsection{Proof of Theorem \ref{Main Theorem} \eqref{Main Theorem 3} for $\Pi(\lm)$}
Let $J = \{ j_1,\ldots, j_N \} = \{ j_{a,b} \mid 1 \leq a < b \leq n \} \in \clA(\Pi(\lm))$.

\begin{lem}\label{vep(J)}
Let $p \in \{ 1,\ldots,n-1 \}$. Suppose that $N_{c,d} = N_{c,c+n-p}$ for all $1 \leq c < d \leq n$ satisfying $d-c \geq n-p$. Then, we have
$$
\vep_p(J) = (-\wt(J),\alpha_p) - l^J_{p,n} = \sum_{a=1}^p (N_{a,a+n-p} - N_{a,a+n-p-1}),
$$
and
$$
N_{c,d}(\Etil_p^{\max}(J)) = \begin{cases}
N_{c,d} \qu & \IF d-c \neq n-p, \\
N_{c,d-1} \qu & \IF d-c = n-p.
\end{cases}
$$
for all $1 \leq c < d \leq n$.
\end{lem}

\begin{proof}
By our assumption and Proposition \ref{Properties of almost decreasing subset} \eqref{Properties of almost decreasing subset 5}, \eqref{Properties of almost decreasing subset 7} and Corollary \ref{aij and nij}, we have
$$
l^J_{p-i-1,n-i-1} - l^J_{p-i,n-i} = N_{p-i,n-i} - N_{p-i,n-i-1} \geq 0
$$
for all $i = 0,\ldots,p-2$, and
$$
(-\wt(J),\alpha_p) - l^J_{1,n-p+1} = N_{1,n-p+1} - N_{1,n-p} \geq 0.
$$
Hence, it follows that
$$
l^J_{p,n} \leq l^J_{p-1,n-1} \leq \cdots \leq l^J_{1,n-p+1} \leq (-\wt(J),\alpha_p).
$$
Then, using Corollary \ref{Properties of extended alcove paths model} \eqref{Properties of extended alcove paths model 5} and equations above, we compute
$$
\vep_p(J) = (-\wt(J),\alpha_p^\vee) - l^J_{p,n} = \sum_{a=1}^p (N_{a,a+n-p} - N_{a,a+n-p-1}).
$$

The second assertion follows by applying Proposition \ref{Properties of Ep(J)} \eqref{Properties of Ep(J) 5} iteratively.
\end{proof}

\begin{prop}\label{String datum of J}
Let us write $\str_{\bfi_A}(J) = (d_{a,b}(J))_{1 \leq a < b \leq n}$. Then, we have
$$
d_{a,b}(J) = \sum_{m=1}^{b-a} (N_{m,m+n-b+1} - N_{m,m+n-b}).
$$
for all $1 \leq a < b \leq n$.
\end{prop}

\begin{proof}
Let us write $\bfi_A = (i_1,\ldots,i_N)$. For each $k \in \{ 1,\ldots,N \}$, define $J^{(k)} \in \clA(\Pi(\lm))$ by
$$
J^{(1)} := J, \qu J^{(k)} := \Etil_{i_{k-1}}^{\max}(J^{(k-1)}).
$$
For each $p \in \{ 1,\ldots,n-1 \}$, set $k_p := \min \{ k \mid i_k = p \}$. Then, each $k \in \{ 1,\ldots,N \}$ can be uniquely written as $k = k_p + p -q$ for some $p \in \{ 1,\ldots,n-1 \}$ and $q \in \{ 1,\ldots,p \}$. Note that $i_{k_p+p-q} = q$, and that $\vep_q(J^{(k_p+p-q)}) = d_{p-q+1,p+1}(J)$.

We show that
\begin{align}\label{Claim}
N_{c,d}(J^{(k)}) = \begin{cases}
N_{c,d} \qu & \IF d-c < n-p, \\
N_{c,c+n-p-1} \qu & \IF n-p \leq d-c < n-q, \\
N_{c,c+n-p} \qu & \IF d-c \geq n-q
\end{cases}
\end{align}
for all $1 \leq c < d \leq n$. If this is the case, then by Lemma \ref{vep(J)}, we obtain
\begin{align}
\begin{split}
d_{p-q+1,p+1}(J) = \vep_q(J^{(k_p+p-q)}) &= \sum_{a=1}^q(N_{a,a+n-q}(J^{(k_p+p-q)}) - N_{a,a+n-q-1}(J^{(k_p+p-q)})) \\
&= \sum_{a=1}^q(N_{a,a+n-p} - N_{a,a+n-p-1}),
\end{split} \nonumber
\end{align}
as desired.

Let us prove \eqref{Claim} by induction on $k = k_p + p -q$. When $k = 1 = k_1 + 1 - 1$, equation \eqref{Claim} clearly holds. Suppose that $k = k_p + p - q \geq 1$, and equation \eqref{Claim} holds. First, consider the case when $q > 1$. In this case, we have $k+1 = k_p + p - (q-1)$. By Lemma \ref{vep(J)}, we have
\begin{align}
\begin{split}
N_{c,d}(J^{(k+1)}) = N_{c,d}(\Etil_q^{\max}(J^{(k)})) &= \begin{cases}
N_{c,d}(J^{(k)}) \qu & \IF d-c \neq n-q, \\
N_{c,d-1}(J^{(k)}) \qu & \IF d-c = n-q
\end{cases} \\
&= \begin{cases}
N_{c,d} \qu & \IF d-c < n-p, \\
N_{c,c+n-p-1} \qu & \IF n-p \leq d-c < n-(q-1), \\
N_{c,c+n-p} \qu & \IF d-c \geq n-(q-1).
\end{cases}
\end{split} \nonumber
\end{align}
This implies equation \eqref{Claim} for $k+1$.

Next, consider the case when $q = 1$. In this case, we have $k+1 = k_{p+1} + (p+1) - (p+1)$. Again by Lemma \ref{vep(J)}, we have
\begin{align}
\begin{split}
N_{c,d}(J^{(k+1)}) = N_{c,d}(\Etil_1^{\max}(J^{(k)})) &= \begin{cases}
N_{c,d}(J^{(k)}) \qu & \IF d-c \neq n-1, \\
N_{c,d-1}(J^{(k)}) \qu & \IF d-c = n-1
\end{cases} \\
&= \begin{cases}
N_{c,d} \qu & \IF d-c < n-(p+1), \\
N_{c,c+n-(p+1)} \qu & \IF d-c \geq n-(p+1).
\end{cases}
\end{split} \nonumber
\end{align}
This implies equation \eqref{Claim} for $k+1$. Thus, the proof completes.
\end{proof}

\begin{cor}
The map $\clA(\Pi(\lm)) \rightarrow \GT(\lm);\ J \mapsto \bfa(J)$ is an isomorphism of crystals.
\end{cor}

\begin{proof}
By Lemma \ref{sufficient condition on isomorphisms}, it suffices to show that $\str_{\bfi_A}(J) = \str_{\bfi_A}(\bfa(J))$ for all $J \in \clA(\Pi(\lm))$. One can easily verify this equality from Propositions \ref{String datum of GT} and \ref{String datum of J}.
\end{proof}

\subsection{Proof of Theorem \ref{Main Theorem} \eqref{Main Theorem 2} and \eqref{Main Theorem 3}}
Let $\Pi = (A_0,\ldots,A_u) \in \widetilde{\AP}(\lm)$, $J = \{ j_1,\ldots,j_N \} \in \clA(\Pi)$. Let $J'$ be an admissible subset that is obtained from a sequence of Yang-Baxter moves, and $\Etil_{i_k}^l \Etil_{i_{k-1}}^{\max} \cdots \Etil_{i_1}^{\max}(J')$ is an almost $\bfi_A$-decreasing subset for all $1 \leq k \leq N$ and $0 \leq l \leq \vep_{i_k}(\Etil_{i_{k-1}}^{\max} \cdots \Etil_{i_1}^{\max}(J'))$. For each $a < b$, let $j_{a,b} \in J'$ be such that $\beta_{j_{a,b}} = \eps_a-\eps_b$. Also, we write $N_{i,j} = N_{i,j}(J')$.

\begin{prop}
Set $\bfa(J') := (a_{i,j}(J'))_{1 \leq i \leq j \leq n}$, where $a_{i,j}(J') := \lm_i-N_{i,j}$. Then, $\bfa(J') \in \GT(\lm)$. Consequently, there exists a unique $J'' \in \clA(\Pi(\lm))$ such that $N(J') = N(J'')$.
\end{prop}

\begin{proof}
Let $1 \leq i < j \leq n$. Since the number of occurrences of $\eps_i-\eps_j$ in $\Gamma(\Pi)$ is equal to $\lm_{i}-\lm_j+1$, we have $0 \leq N_{i,j} \leq \lm_{i}-\lm_j$. This implies that
$$
\lm_j \leq \lm_i - N_{i,j} \leq \lm_i.
$$

Let $1 \leq a < b < c \leq n$. Since $(a,b) <_{\bfi_A} (a,c) <_{\bfi_A} (b,c)$ and $(\beta_{a,b},\beta_{a,c}), (\beta_{a,c},\beta_{b,c}) \neq 0$, we have $j_{b,c} < j_{a,c} < j_{a,b}$. From Proposition \ref{Characterization of extended chain} and equation \eqref{l, N, and lm}, it is easy to see that
$$
-l_{a,c} \geq -l_{b,c}, \AND N_{a,c} \geq N_{a,b}.
$$
These inequalities imply that
$$
\lm_b - N_{b,c} \leq \lm_a - N_{a,c} \leq \lm_a - N_{a,b}.
$$
Then, for each $1 \leq k \leq n$ and $1 \leq i < j \leq n$, we have
$$
\lm_k - N_{k,k} = \lm_k,
$$
and
$$
\lm_{i+1} - N_{i+1,j} \leq \lm_i - N_{i,j} \leq \lm_i - N_{i,j-1}.
$$
This shows that $\bfa(J') = (\lm_i - N_{i,j})_{1 \leq i \leq j \leq n}$ is a Gelfand-Tsetlin pattern of shape $\lm$, as desired.
\end{proof}

\begin{prop}\label{String datum of J for general Pi}
Let $J'' \in \clA(\Pi(\lm))$ be such that $N(J') = N(J'')$. Then, we have $\str_{\bfi_A}(J') = \str_{\bfi_A}(J'')$.
\end{prop}

\begin{proof}
In Proposition \ref{String datum of J}, we described $\str_{\bfi_A}(J'')$ in terms of $N(J'')$. During the calculation, we used the facts that $\Etil_{i_k}^l \Etil_{i_{k-1}}^{\max} \cdots \Etil_{i_1}^{\max}(J'')$ is an almost $\bfi_A$-decreasing subset for all $1 \leq k \leq N$ and $0 \leq l \leq \eps_{i_k}(\Etil_{i_{k-1}}^{\max} \cdots \Etil_{i_1}^{\max}(J''))$, but did not use the fact that $J'' \in \clA(\Pi(\lm))$. Therefore, we can describe $\str_{\bfi_A}(J')$ in terms of $N(J')$ in exactly the same way as $\str_{\bfi_A}(J'')$. Since $N(J') = N(J'')$, we conclude that $\str_{\bfi_A}(J') = \str_{\bfi_A}(J'')$.
\end{proof}

Now, let us prove Theorem \ref{Main Theorem} \eqref{Main Theorem 2}. Recall that $J'$ is obtained from $J$ by a sequence of Yang-Baxter moves, which is a crystal isomorphism. Hence, we have
$$
\str_{\bfi_A}(J) = \str_{\bfi_A}(J') = \str_{\bfi_A}(J'').
$$
On the other hand, $\str_{\bfi_A}(J'')$ determines $N(J'')$ since the assignment $\tilde{J} \mapsto N(\tilde{J})$, $\tilde{J} \in \clA(\Pi(\lm))$ is injective. Therefore, $N(J')(=N(J''))$ is determined by $\str_{\bfi_A}(J)(=\str_{\bfi_A}(J''))$, which is independent of the choices of $J'$. Thus, Theorem \ref{Main Theorem} \eqref{Main Theorem 2} follows.

Finally, we prove Theorem \ref{Main Theorem} \eqref{Main Theorem 3}. The assignment $J \mapsto J''$ above is an isomorphism $\clA(\Pi) \rightarrow \clA(\Pi(\lm))$ by Lemma \ref{sufficient condition on isomorphisms}. On the other hand, we have an isomorphism $\clA(\Pi(\lm)) \rightarrow \GT(\lm)$ of crystals which sends $J''$ to $\bfa(J'')$. Since $N(J) = N(J') = N(J'')$, we have $\bfa(J) = \bfa(J') = \bfa(J'')$. Therefore, the assignment $J \mapsto \bfa(J)$ is a composite of two isomorphisms of crystals, and hence, an isomorphism. This proves Theorem \ref{Main Theorem} \eqref{Main Theorem 3}.


\begin{thebibliography}{99}
\bibitem[BB05]{BB05} A. Bj\"{o}rner and F. Brenti, Combinatorics of Coxeter Groups, Graduate Texts in Mathematics, 231. Springer, New York, 2005. xiv+363 pp.

\bibitem[BFP99]{BFP99} F. Brenti, S. Fomin, and A. Postnikov, Mixed Bruhat operators and Yang-Baxter equations for Weyl groups, Internat. Math. Res. Notices 1999, no. 8, 419--441. 

\bibitem[BS17]{BS17} D. Bump and A. Schilling, Crystal Bases, Representations and combinatorics. World Scientific Publishing Co. Pte. Ltd., Hackensack, NJ, 2017. xii+279 pp.

\bibitem[H90]{H90} J. E. Humphreys, Reflection Groups and Coxeter Groups, Cambridge Studies in Advanced Mathematics, 29. Cambridge University Press, Cambridge, 1990. xii+204 pp.

\bibitem[K90]{K90} M. Kashiwara, Crystalizing the $q$-analogue of universal enveloping algebras, Comm. Math. Phys. 133 (1990), no. 2, 249--260. 

\bibitem[K91]{K91} M. Kashiwara, On crystal bases of the $Q$-analogue of universal enveloping algebras, Duke Math. J. 63 (1991), no. 2, 465--516. 

\bibitem[L07]{L07} C. Lenart, On the combinatorics of crystal graphs. I. Lusztig's involution, Adv. Math. 211 (2007), no. 1, 204--243. 

\bibitem[LL15]{LL15} C. Lenart, A. Lubovsky, A generalization of the alcove model and its applications, J. Algebraic Combin. 41 (2015), no. 3, 751--783. 

\bibitem[LP07]{LP07} C. Lenart and A. Postnikov, Affine Weyl groups in 
K-theory and representation theory, Int. Math. Res. Not. IMRN 2007, no. 12, Art. ID rnm038, 65 pp. 

\bibitem[LP08]{LP08} C. Lenart and A. Postnikov, A combinatorial model for crystals of Kac-Moody algebras, Trans. Amer. Math. Soc. 360 (2008), no. 8, 4349--4381. 

\bibitem[Lu90]{Lu90} G. Lusztig, Canonical bases arising from quantized enveloping algebras, J. Amer. Math. Soc. 3 (1990), no. 2, 447--498. 
\end{thebibliography}
\end{document}